\documentclass[12pt,reqno]{amsart}
\usepackage[T1]{fontenc}  
\usepackage[latin1]{inputenc} 
\usepackage{amssymb}
\def\Q {{\mathbb Q}}

\def\cC {{ \mathcal C}}

\def\cG {{ \mathcal G}}
\def\cH{{ \mathcal H}}
\def\cT {{ \mathcal T}}
\def\cO {{ \mathcal O}}

\def\R {{\mathbb R}}
\def\N {{\mathbb N}}

\def\F {{\mathbb F}}
\def\cF {{\mathcal F}}
\def\cS {{\mathcal S}}
\def\cR {{\mathcal R}}
\def\Z {{\mathbb Z}}

\def\K {{\mathbb K}}

\def\cZ {{\mathcal Z}}
\def\C {{\mathbb C}}

\newtheorem{theorem}{Theorem}

\newtheorem{proposition}{Proposition}
\newtheorem{corollary}{Corollary}
\newtheorem{lemma}{Lemma}
\newtheorem{assumption}{Assumption}

\newtheorem{definition}{Definition}
\newtheorem{remark}{Remark}
\newtheorem{example}{Example}

\theoremstyle{remark}
\def \b {\bigskip}
\def \m {\medskip}
\def \s {\smallskip}

\begin{document}
\date{\today}

\title[Commutative Banach perfect semi-fields of characteristic $1$.]{A classification of the commutative Banach perfect semi-fields of characteristic $1$:  Applications.} 

\centerline{\author{ Eric Leichtnam}}

\address{CNRS Institut de Math\'ematiques de Jussieu-PRG, b\^atiment Sophie Germain.}
\email{eric.leichtnamATimj-prg.fr}

\maketitle

\m
{\it A la m\'emoire de Tilby.}

\b

\noindent \begin{abstract} We define and study the concept of commutative Banach perfect semi-field 
$(\cF, \oplus, +)$ of characteristic $1$. The metric allowing to define the Banach structure comes from Connes \cite{C}.  We define the  spectrum $S_E(\cF)$ as the set 
of normalized characters $\phi: (\cF, \oplus, +)  \rightarrow (\R, \max ( , ) , +)$. 
 This set is shown to be naturally a compact space. 
Then we construct  an isometric isomorphism of Banach semi-fields of Gelfand-Naimark type: 
 
$$
\Theta: (\cF, \oplus, +)  \rightarrow (C^0(S_E(\cF), \R), {\max}\,( ,), +)
$$
$$
X \mapsto \Theta_X: \phi \mapsto \phi(X)= \Theta_X (\phi)\,.
$$ In this way, $\cF$ is naturally identified with the set of real valued continuous functions on $S_E(\cF)$.
Our proof 
 relies on a study of the congruences on $\cF$ and on a new Gelfand-Mazur type Theorem. As a first application, we deduce that the  spectrum 
 of the Connes-Consani Banach algebra of the Witt vectors of $(\cF, E)$ coincides with $S_E(\cF)$.
 We give many other applications.  Then we  study the case of the commutative cancellative perfect  semi-rings 
$(\cR, \oplus, +)$ and also  give   structure theorems in the Banach case. Lastly, we use these results to propose the foundations 
of a new scheme theory in the characteristic $1$ setting. We introduce a topology of Zariski type on $S_E(\cF)$ and the concept of valuation associated 
to a character $\phi \in S_E(\cF)$. Then we  come to the notions of $v-$local semi-ring and of scheme.
\end{abstract}

\tableofcontents

\section{ Introduction.} $\;$

A semi-ring $(\cR, \oplus, +)$ of characteristic $1$ is  a set $\cR$ endowed with two commutative and associative laws satisfying the following conditions. The law $\oplus$ is    idempotent 
(e.g. $X \oplus X= X,\,\forall X \in \cR$) whereas the   additive law $+$ has a neutral element $0$ and is distributive with respect to $\oplus$. 
The semi-ring $\cR$ is said to be cancellative if $X+Y=X+Z$ implies $Y=Z$.  If 
for any $n \in \N^*, \, X \mapsto n X$ is surjective from $\cR$ onto itself, then $(\cR, \oplus, +)$ is  called a perfect semi-ring.
If 
$(\cR, +)$ is a group then $(\cR, \oplus, +)$ is called a  semi-field and we denote it by $\cF$ rather than $\cR$.

Morally, the first law $\oplus$ is the new addition and the second law $+$ is the new multiplication. But, being idempotent, 
$\oplus$ is very far from being cancellative.

The theory of semi-rings is well developed 
(see for instance \cite{Golan}, \cite{Pareigis}). It has important applications in various areas: tropical geometry (\cite{Ak}, \cite{Br}), computer science, tropical algebra (\cite{Iz}, \cite{Lescot}), mathematical physics (\cite{Li}). In the references of these papers the reader will find 
many  other interesting  works in these areas.

Thanks to Connes-Consani, the semi-ring theory of characteristic $1$  plays also  a very interesting role in Number Theory: 
\cite{CC1}, \cite{CC2}, \cite{CC3}, \cite{CC4}. For instance they constructed an Arithmetic Site (for $\Q$) which encodes the Riemann Zeta Function. They analyzed the contribution of the archimedean place of $\Q$ through the 
semi-field $(\R, {\max}\, ( ,), +)$, the multiplicative action of $(\R^{+*}, \times)$  on $\R$ being an analogue of the Frobenius. 

In constructing Arithmetic sites of Connes-Consani type for other numbers fields $K$, 
 Sagnier( \cite{Sa}) was naturally lead to more general semi-fields of characteristic $1$. They are  associated to some tropical geometry.
For example, in the case $K=\Q[i]$, \cite{Sa} introduced the set $\cR_c$   of compact convex (polygons) of $\C$ which are invariant under multiplication by $i$. Then he 
 analyzed the contribution of the archimedean  place of $\Q[i]$
through the semi-field of fractions of the cancellative perfect semi-ring:
$$
(\cR_c, \oplus= conv\,( , ), + )\,.
$$ 

\s

Furthermore, Connes-Consani (\cite{CC1}, \cite{C}) have developed a very interesting theory of Witt vectors in characteristic $1$. Following \cite{C} one defines the analogue of a $p-$adic metric on a perfect semi-field $\cF$ in the following rough way (see Section 2 for details).
First recall  the well-known partial order on $\cF$
 defined by: $X \leq Y $ if $X\oplus Y= Y$. 
Next, it turns out that $\cF$ carries a natural structure of $\Q-$vector space.  
 Assume then the existence of $E \in \cF$ such that for any $X \in \cF$ there exists  $t  \in \Q^+$ 
such that $-t E \leq X \leq tE$ (see Assumption \ref{ass:1}). 
Denote by $r(X)\in \R^+$  the infimum of all such $t \in \Q^+$ and assume moreover that $r(X)=0$ implies  $X=0$ (see Assumption \ref{ass:2}).
Then $(X,Y) \mapsto r(X-Y)$ defines a distance on $\cF$; we say that 
$(\cF, \oplus, +)$ is a Banach semi-field if $\cF$ is complete for this distance.


The norm $r(X)$ satisfies also the following inequality of non-archidemean type:
$$
\forall X, Y \in \cF, \; r(X \oplus Y) \leq \max ( r(X) , r(Y) )\,.
$$
Thus there is also an analogy 
between the  concept of commutative perfect Banach semi-field of characteristic $1$ and the one of perfect nonarchimedean Banach field. 
The latter has been investigated by 
  deep works in the nonarchimedean world (\cite{Be}, \cite{Kedlaya},  \cite{KL}), 

\s
{\it 
Therefore, it seems relevant to  try first to
 classify  the  commutative Banach perfect semi-fields $(\cF, \oplus, +)$ of characteristic $1$. }
To this aim, it is natural to associate to $(\cF, \oplus, +)$ the set $S_E(\cF)$ of  the characters $\phi: (\cF, \oplus, +) \rightarrow (\R, \max ( \,) , +)$
such that $\phi(E)=1$. More precisely, $\phi$ satisfies the following:
$$
\forall X, Y \in \cF,\; \phi(X+Y)= \phi(X) + \phi(Y),\; \phi(X \oplus Y)= \max ( \phi(X) , \phi(Y) )\,.
$$
We endow $S_E(\cF)$ with a natural topology $\cT$ making it a compact space, we may call it the spectrum of $\cF$.
Then we construct (see Theorem \ref{thm:Tilby} for details) an isometric isomorphism of Banach semi-fields of Gelfand-Naimark type: 
 
\begin{equation} \label{eq:Tilby}
\Theta: (\cF, \oplus, +)  \rightarrow (C^0(S_E(\cF), \R), {\max}\,( ,), +)
\end{equation}
$$
X \mapsto \Theta_X: \phi \mapsto \phi(X)= \Theta_X (\phi)\,.
$$  Thus $X\in \cF$ is identified with the continuous function $\Theta_X$ on the spectrum $S_E(\cF)$.

\m
Now we describe briefly the content of this paper. 
 
 In Section 2.1 we define the concept of commutative perfect semi-field $(\cF, \oplus, +)$
of characteristic $1$  and introduce some basic material. Actually we need to introduce  on $\cF$
 the concept of $F-$norm which is more general than the one given by $r(X-Y)$. By definition a $F-$norm $\| \, \|$ 
 satisfies:
 $$
 \| X \oplus Y - X' \oplus Y' \| \, \leq \, \max ( \| X-X'\| , \| Y-Y'\| )\,.
 $$
If $(\cF, \| \, \|)$ is complete then,  guided by \cite{C}, we construct a continuous Frobenius action of 
$(\R^{+*}, \times)$ on $\cF$ and we obtain a natural real vector space structure on $\cF$. We also prove various Lemmas showing 
that $r(X)$ has morally the properties of a spectral radius.  

In Section 2.2 we introduce the set $S_E(\cF)$  of the normalized characters  and endow it with the weakest topology (called $\cT$) rendering continuous all the maps $\phi \mapsto \phi(X)$, $X \in \cF$. 
In Theorem \ref{thm:char} we prove that $(S_E(\cF), \cT)$ is a compact topological space. 

In Section 3 we define and study algebraically the notion of congruence  for $(\cF, \oplus, +)$. 
A  congruence $\sim$ is the analogue in characteristic $1$ of the concept of ideal in classical Ring theory.
The quotient map  $\pi: \cF \rightarrow \cF /\sim$ defines a homomorphism of semi-fields of characteristic $1$. If $\sim$ is closed then the norm $r(X-Y)$ induces a $F-$norm on the quotient $\cF /\sim$ which is stronger  than the  norm $r_\sim (\pi(X)-\pi(Y))$ associated to $\pi(E)$.
When  $\cF$ is complete for $r(X-Y)$,  it does not seem possible to see (directly) that $\cF /\sim$  is complete for the norm $r_\sim (\pi(X)-\pi(Y))$. This is why we introduced the concept 
of $F-$norm which, in some sense, is stable by passage to the quotient.
 
Maximal congruences $\sim$ are shown to be closed, and a Gelfand-Mazur  type theorem is proved for such ones (see Theorem \ref{thm:gm}):  
$$( \cF/ \sim, \oplus, +)  \simeq (\R, \max( \, , ), +)\,.
$$   The proof uses a suitable spectral theory for the elements of $\cF$, obtained from the partial order of $(\cF, \oplus, +)$.
Notice that, in this setting, there is no available theory of holomorphic functions. 
Then we show that for each $X \in \cF$ there exists  $\phi \in S_E(\cF)$, $| \phi(X) |= r(X)$. In particular 
the set $S_E(\cF)$ is not empty.

In Section 4 we state and prove our classification result Theorem \ref{thm:Tilby}. We also prove that 
if $\cF$ is  complete for a $F-$norm  then the map $\Theta$ of \eqref{eq:Tilby} is only an injective continuous 
homomorphism with dense range.

As a first application, we consider the real Banach algebra of Witt vectors $\overline{W}( \cF, E)$ that  Connes-Consani (\cite{CC1}, \cite{C}) constructed  
functorially with respect to $(\cF, E)$. The authors observed that its norm is not $C^*$.
Nevertheless, they   pointed out that the spectrum of  $\overline{W}( \cF, E)$ should contain some interesting information for $(\cF, \oplus, +)$. 
 By applying first our  Theorem \ref{thm:Tilby} and then \cite{C}[Prop.6.13], 
we prove that the spectrum of $\overline{W}( \cF, E) \otimes \C$ coincides with $S_E(\cF)$, see  Theorem \ref{thm:Connes}. Actually the Connes-Consani's construction 
of Witt vectors is done in a slightly broader context. It would be interesting to explore the connexion between our approach and \cite{C}, we shall try to bring a contribution to this in a separate paper.

We give several other applications of Theorem \ref{thm:Tilby}. 
Geometric characterization of the elements of $\cF$ which are $\geq 0$, regular, absorbing.
Determination of the characters of $(C^0(K,\R), \max( ,), +)$, $K$ being  a compact space. Determination of the characters of the semi-ring $\cC$ of the convex compact subsets of $\R^n$ which contain the origin $0$.  
More precisely   consider a map  $\phi : \cC \rightarrow \R$  such that:
$$
\forall A,B \in \cC,\; \phi ( {\rm conv} \, (A \cup B) ) = \max ( \phi ( A), \phi (B)) ,\; \phi(A+B) = \phi(A) + \phi (B)\,,
$$ where ${\rm conv} \, (A \cup B)$ denotes the convex hull.
Then we show the existence of $\psi \in (\R^n)^*$ such that $\forall A \in \cC,\; \phi(A) = \max_{v\in A} \psi(v)$.
In other words, the spectrum $S_E( \cC)$ of the semi-ring $(\cC , \oplus= {\rm conv} ( \cdot  \cup  \cdot ), +) $  is an euclidean sphere of $(\R^n)^*$.

\s
In Section 5 we determine the closed congruences of $(\cF, \oplus , +)$ by using Theorem \ref{thm:Tilby}, 
 its proof and Urysohn. Via Theorem \ref{thm:Tilby}, they correspond bijectively to the closed subsets of $S_E(\cF)$: see Theorem \ref{thm:Tilby'}.

\s
In Section 6 we consider a commutative cancellative perfect semi-ring of characteristic $1$ $(\cR, \oplus, +)$, 
see Definition \ref{def:gen}. We apply the tools constructed in the previous sections to the study of $\cR$.

We first establish a nice one to one correspondence between cancellative congruences on $\cR$ and congruences on its semi-field of fractions $\cF$.
 Then, following again
 \cite{C},
we make Assumptions
\ref{ass:c}  and \ref{ass:d} which allow to define the distance $r(X-Y)$. Denote by $S_E(\cR)$  the (compact) set of normalized characters of  $\cR$.
 We prove that the map:
\begin{equation} \label{eq:embR}
j\, :\, (\cR, \oplus, +) \rightarrow (C^0(S_E(\cR), \R), \max( ,), +) 
\end{equation} 
$$
X \mapsto j_X: \phi \mapsto \phi(X)=\, j_X (\phi)
$$
defines an  injective homomorphism of  semi-rings. See Theorem \ref{thm:mapj}.
   
Next we assume that $\cR$ is Banach (i.e complete for the distance $r(X-Y)$) and consider a maximal cancellative congruence $\sim$ on $\cR$. We then prove 
a Gelfand-Mazur type Theorem (see Theorem \ref{thm:gmring}). Precisely,  if all the elements of the quotient $\cR / \sim$ 
are $\geq 0$ then one has:
$$
(\cR/ \sim , \oplus, +) \, \simeq \, (\R^+ , \max ( \, , ) , +)\,; 
$$ otherwise one has $\cR/ \sim \,\, \simeq \R$.

 Lastly, in the special case $\cR = \{ X \in \cF / \, 0 \leq X \}$, 
 we show that  $\cF$  is Banach and that the range of the map \eqref{eq:embR} is exactly $C^0(S_E(\cR) , \R^+)$. See Theorem 
\ref{thm:banring}.


\s

In Section 7, we have algebraic goals in mind and 
use our previous results to propose the foundations of a new scheme theory in characteristic $1$. Notice that 
several interesting scheme theories, different from ours,  already exist in the tropical setting:  \cite{Iz}, \cite{Lor}, \cite{To}.
 We consider a cancellative commutative semi-ring $(\cR, \oplus, +)$, with semi-field of fractions $\cF$, satisfying 
 conditions stated in Assumption \ref{ass:last}. 

Our goal is to
construct algebraic tools allowing  to decide whether 
an element of   $\cF$ belongs or not to $\cR$. We would like also to detect sub semi-rings of $\cR$ which have some arithmetic flavor.
A basic example is provided by the semi-ring $\cR_c( [0,1])$ of the piecewise affine convex functions on $[0,1]$ 
and its semi-field of fractions $\cF_c( [0,1])$. An example of  sub semi-ring of $\cR_c( [0,1])$ having an arithmetic flavor 
  is given by the set of  functions $ x \mapsto \max_{1 \leq j \leq n} (a_j x + b_j)$ 
  where the $a_j, b_j$ belong to a sub field $\K$ of $\R$.



Our previous results enable us to   define  a topology $\cZ$ on $S_E( \cF)$ and prove that it satisfies  properties of  Zariski type. (See \cite{Lescot} for a similar Zariski topology  but on a different set). Then, motivated by the case of the semi-ring of piecewise affine functions 
on $[0,1]$, we define for each $\phi \in S_E(\cR)= S_E(\cF)$ the concept of valuation 
$V_\phi : \phi^{-1}\{0\} \cap \cR \rightarrow \R^+$. This map is additive and satisfies $\max ( V_\phi(X), V_\phi(Y) ) \leq V_\phi( X \oplus Y)$.
Next, guided by an   analogy with classical algebraic geometry, we define the concepts of $v-$local semi-ring $(\cR, \phi, V_\phi)$
 and of localization along $(\phi, V_\phi)$. 

 Then we introduce the notion of localization data on $S_E(\cF)$,  which enables us to globalize these constructions in the framework of sheaf theory. Their existence is an hypothesis, 
 it allows to define  
 a structural sheaf 
$\cO$ on $S_E( \cF) $, this leads to the concept of affine scheme $(S_E(\cF), \cO)$ in characteristic $1$. More generally we define the notion 
of $v-$locally semi-ringed spaces $(S , \mathcal{O})$; it is a scheme if in addition it is locally isomorphic to an open subset of an affine scheme. A natural example 
is provided by $\R/ \Z$ and the piecewise affine functions.

\m

V. Kala kindly pointed out to us the references \cite{Ka}[Thm.4.1] and  \cite{Mu}[Thm.5.1]. There, the authors have proved interesting structure theorems  for finitely generated semi-fields which admit 
an order-unit but do not satisfy necessarily our Assumption \ref{ass:2}.
Of course a perfect semi-field is far from being finitely generated. Nevertheless, e-mail exchanges with V. Kala suggest that the comparison of 
the methods of this paper with the ones of \cite{Ka} and \cite{Mu} should lead to something interesting.  

\b
\noindent {\bf Acknowledgement.} The author thanks S. Simon and G. Skandalis  for interesting 
 conversations and Etienne Blanchard for helpful comments. Part of this work was done while the author was visiting the University of Savoie-Mont Blanc, he would like to thank 
P Briand for the kind hospitality. Finally, the memory of Tilby has been  a great motivation to achieve this work.

\s
\section{Definitions, first properties and the spectrum $S_E(\cF)$.} $\;$

\subsection{Definition of a commutative Banach perfect semifield of characteristic $1$.} $\;$

\m

We first define precisely the class of semi-fields we shall be interested in.
\begin{definition}  \label{def:1} A  commutative perfect semifield of characteristic $1$ 
is  a triple $(\cF, \oplus, +)$ where $\cF$ is a set endowed with two laws $\oplus, +$  such that the
 following conditions are satisfied: 
 \item 1] $(\cF, \oplus)$ is a commutative and associative monoid such that $\forall X \in \cF, X \oplus X= X$.
\item 2] $(\cF, +)$ is an abelian group (with identity element $0$) such that 
$$
\forall X,Y, Z \in \cF,\; X+ (Y\oplus Z)= (X+Y) \oplus (X+Z)\,.
$$
\item 3] For any $n \in \N^*$, $\theta_n:\, X \mapsto n X$ is a (set theoretic) surjective map from $\cF$ onto $\cF$, where 
$n X = X + \ldots + X$, $n$ times.
\end{definition}
\begin{remark} Usually, one adds to $\cF$ a (ghost) element denoted $-\infty$ and require that for all $X \in \cF,\, 
-\infty \oplus X= X$ and $-\infty + X= -\infty$. But, since this element will play no role for our goal, we shall forget it. Therefore,
$\oplus$ has no identity element in our framework.
\end{remark}
\s By a straightforward and well known computation, the conditions of Definition 1, imply:
\begin{equation} \label{eq:computation}
\forall (n, X, Y) \in \N^*\times \cF \times \cF ,\; \; n( X \oplus Y ) =  \displaystyle{ \oplus_{k=0}^n\,\bigl( k X + (n-k) Y \bigr)}\,.
\end{equation}
Given the cancellative property of the group law $+$ and Definition \ref{def:1}.3], one has the following fundamental 
result.

\begin{proposition} \label{prop:Go} (\cite{Golan}[Prop. 4.43], \cite{C}[Lemma 4.3]). Let $(\cF, \oplus, +)$ be as in the previous Definition.
Then:
\item 1] For any $n \in \N^*$, the map $\theta_n:\, X \mapsto n X$ is injective (and thus bijective) and induces 
an isomorphism of $(\cF, \oplus, +)$.
\item 2] For any $n \in \N^*$, the map $\theta_{-n}: \, X \mapsto (-n)X= -(nX)$ is bijective from $\cF $ onto $\cF $ 
and induces an isomorphism of the additive group $(\cF, +)$.
\end{proposition}
It is worth to outline that the injectivity of $\theta_n$ is  a fundamental  fact which uses 
a subtle interplay between $\oplus$, $+$ and the cancellativity property of $+$. Moreover, the following property 
is also non trivial.
$$
\forall (n, X, Y) \in \N^* \times \cF \times \cF,\; \theta_n(X \oplus Y)= \theta_n(X) \oplus \theta_n(Y)\,.
$$ The starting idea of the proof is to  observe that \eqref{eq:computation} implies the following:
$$
( n X \oplus n Y) \, + \,(n-1) ( X \oplus Y) \, = \, (2n-1) (X \oplus Y)\, = n (X \oplus Y) + (n-1) (X \oplus Y).
$$

Notice of course that $-(X \oplus Y) $ {\bf is not equal to} $-X \oplus -Y$. Indeed, see 
Lemma \ref{lem:oplus}.4].

\m

The next Lemma explains how the conditions of Definition \ref{def:1} allow to endow $\cF$ with a natural 
structure of $\Q-$vector space.

\begin{lemma} \label{lem:Frob} 
\item 1] The equality 
$$\theta_{a/b}= \theta_a \circ \theta_b^{-1}, \; (a,b) \in \N^* \times \N^*\,,$$
defines an action of $\Q^{+*}$ on $(\cF, \oplus, +)$ satisfying:
$$
\forall (t,t', X)\in \Q^{+*} \times \Q^{+*} \times \cF,\; \,
\theta_{t t'}= \theta_t \circ \theta_{t'} ,\;  
\theta_t(X) + \theta_{t'}(X) = \theta_{t + t'}(X)\,.
$$ 
\item 2]  Denote by $\theta_0$ the (zero) map sending $X \in \cF$ to $0\in \cF$ and, 
  for all $(t,X)\in \Q^{+*}\times \cF$ set:
$$
\theta_{-t}(X)= -\theta_t (X) =\theta_t (-X) \,.
$$
 Then the maps 
$\theta_0$ and  $\theta_t$, $t \in \Q$
 endow $(\cF, +)$ with a structure 
of  $\Q-$vector space.
\end{lemma}
\begin{proof} 1] This is proved in \cite{C}[Prop. 4.5]. 

\noindent 2] Given part 1],  it suffices to check the following for  all $t, t' \in \Q$ such that 
$0 \leq t\leq t'$:
\begin{equation} \label{eq:vect}
\forall X \in \cF,\;  \theta_{t'-t} (X) = \theta_{t'}(X) -\theta_{t}(X),\; \theta_{t-t'} (X) = 
\theta_{t}(X) -\theta_{t'}(X)\,.
\end{equation} The second 
equality of \eqref{eq:vect} follows from the first by using the inverse for $+$.
The equality $ \theta_{t'-t} (X) = \theta_{t'}(X) -\theta_{t}(X)$ is equivalent to 
$  \theta_{t'-t} (X) + \theta_{t}(X)= \theta_{t'}(X)$. But  this is true thanks to part 1].  
\end{proof}
\s It is time now to simplify the notation.
\begin{definition} Let $t\in \Q$ and $X\in \cF$. In the sequel, we shall set
$t X = \theta_t(X)$.  But  one should keep in mind that
for $t\in \Q^{+*}$, $X \mapsto t X$ defines the analogue of the Frobenius action in characteristic $p$ (see \cite{C}[Prop.4.5]).

\end{definition}
\s
We now recall the  definition of the canonical partial order of $(\cF , \oplus, +)$.

\begin{definition} \label{def:order} Any  semifield $\cF$ as in Definition \ref{def:1} is endowed with the partial order $\leq$ defined by:
$$
\forall X, Y \in \cF,\; X \leq Y \; {\rm iff}\; X\oplus Y = Y\,.
$$ The set of positive  ($\geq 0$) elements, $\{X\in \cF /\,\, 0\oplus X=X\}$, defines a perfect semiring which is cancellative.

\end{definition}

\m
Next, we give two explicit examples of such $(\cF, \oplus, +)$.
 
\begin{example} The two most elementary examples of perfect semifields (as in Definition \ref{def:1}) are given by 
$(\Q, {\max}\,( , ) , + )$ and
$(\R, {\max}\,( , ) , + )$. The partial order is here of course the usual one.
\end{example}
\begin{example} Endow the set $\cF_{paf}$ of piecewise affine functions $[0,1] \mapsto \R$  with 
the  law defined by
 $$\forall X,Y \in \cF,\;
 \forall t \in [0,1],\; (X \oplus Y)(t)= {\max}\,(X(t),Y(t))\,.$$ Denote by $+$  the usual addition of functions: $X+Y$.  Then 
$(\cF_{paf}, {\max}\,( , ) , + )$   is a commutative 
perfect semifield of characteristic $1$. Moreover, with the notations of Definition \ref{def:order},  $X\leq Y$ if and only if 
$\forall t \in [0,1], \; X(t) \leq Y(t)$.
\end{example}

\m
In the next Lemma we collect a few properties of this partial order $\leq$.
\begin{lemma} \label{lem:oplus} Let $X,Y,Z \in \cF$ and $t \in \Q^{+*}$. Then the followings are true:
\item 1] If $X \leq Y$ then one has:
$$X + Z \leq Y +  Z,  \; \,X \oplus Z \leq Y \oplus  Z,\;\, t X \leq t Y\,.
$$
\item 2] $ X \leq Y \; \Leftrightarrow \; 0 \leq Y-X \; \Leftrightarrow\;  -Y \leq -X$.
\item 3] $X= (0\oplus X) - (0 \oplus -X)$.
\item 4] $X+Y= X\oplus Y -( -X \oplus -Y)$.
\end{lemma}
\begin{remark} The mental picture to have in mind is that $X \oplus Y$ is the $\max$ of $X$ and $Y$, whereas 
$-( -X \oplus -Y)$ is the minimum. Moreover, Part 3] shows that any $X\in \cF$ admits a canonical 
decomposition as the sum of its positive ($\geq 0$) part and negative ($\leq 0$) part.
\end{remark}
\begin{proof} 1] By Proposition \ref{prop:Go}.1], one has:
$t( X \oplus Y) = t Y = t X \oplus t Y$. This shows that $t X \leq t Y$. The rest of 1] as well as 2] are 
well-known and left to the reader.

3] Using the distributivity of $+$, one sees that the identity to be proved is equivalent to:
$$
0 = -X + (0 \oplus X) \;- (0 \oplus - X)= (-X +0) \oplus (-X + X)\; - (0 \oplus - X)\,.
$$ Since  the right hand side of the second equality is clearly zero, the result is proved. 4] can be proved similarly and is left to the reader.
\end{proof}

\s
The following Assumption requires the existence of an element $E$ which allows, in some sense, to absorb 
any other element of $\cF$.
\begin{assumption} \label{ass:1} There exists $E \in \cF$ such that  $\forall X \in \cF$, 
$$\exists t \in \Q^+,\,-t E \leq X \leq t E\,.$$
 Denote by $r(X) \in \R^+$ the infimum of  all  such $t\in \Q^+$.
\end{assumption}

\s
The next Lemma  shows that such an $E$ has to be $\geq 0$ and that $(X,Y) \mapsto r(X-Y)$ defines a pseudo-distance.

\begin{lemma} \label{lem:dist} Suppose that the previous Assumption is satisfied. Then the following are true.

\item 1] Let $X \in \cF$ such that $0 \leq X$. Then:
$$\forall t,t' \in \Q^{+},\;  t\leq t' \Rightarrow t X \leq t' X \,.$$
\item 2] One has necessarily: $0\leq E$.
\item 3] $\forall X, Y \in \cF$, $r(X)=r(-X)$ and $r(X+Y) \leq r(X) + r(Y)$.

\item 4] $\forall (t, X) \in \Q^+ \times \cF,\; r(t X) = t \,r (X)$.
\end{lemma}
\begin{proof}1] Observe that $0 \leq (t'-t) X $ by Lemma \ref{lem:oplus}.1]. 
Then, one has:
$$
t' X \oplus t X\, = \, ( (t'-t )X + t X ) \oplus t  X= (t'-t)X \oplus 0 \; \, + \,
t X=
$$
$$(t'-t) X +t X=  t' X \,.
$$ This proves the result.

 2]  Write 
$$E_+ = 0 \oplus E,\,\; E_-=0 \oplus (-E)\,.$$
By  Lemma \ref{lem:oplus}.3], one has  $E= E_+ - E_-$. By Assumption 1 (and part 1])
there exists $t \in\Q^+$
such that $ E_+ + E_- \leq t E_+ - tE_- \leq (t+2) E_+ - tE_-$. This implies  that:
$$
(t+1) E_- \leq  (t+1) E_+\,.
$$ Using Lemma \ref{lem:oplus}.1], one obtains: $E_- \leq E_+$. Then 
Lemma \ref{lem:oplus}.2] implies that 
$$0 \leq E_+ - E_-= E\,.$$ This proves the result. Lastly, 3] and 4] are easy consequences of Lemma \ref{lem:oplus}
and are left to the reader.
\end{proof}
\s
\begin{assumption} \label{ass:2} For any $X \in \cF$, $r(X)=0$ if and only if $X=0$.
\end{assumption}

\s 
{\bf Until the end of Section 5  we shall suppose  that $(\cF, \oplus, +)$ is a semi-field $(\not=\{0\})$ as in Definition \ref{def:1} 
which satisfies the two  Assumptions \ref{ass:1} and \ref{ass:2}. } 

\s
  Lemma \ref{lem:dist} shows that $(X,Y) \mapsto r(X-Y)$
defines a distance on $\cF$. As pointed out in  \cite{C}[Sect.6], this distance is an analogue, in the characteristic $1$ setting,  of a $p-$adic distance. 
The following Lemma allows to conclude that,  finally, all the algebraic operations of $\cF$ are continuous for this distance. It also proves 
an inequality of ultra-metric (or non archimedean) type for the first law $\oplus$.

\begin{lemma} \label{lem:cont} 
\item 1] 
For any $X,X',Y,Y' \in \cF$ one has:
$$
r( X \oplus Y - X' \oplus Y') \, \leq \, \max ( r(X-X') , r(Y-Y') )\,.
$$

\item 2] For all $(X,Y) \in \cF\times \cF $, one has:
$r(X \oplus Y) \leq {\max}\, ( r(X), r(Y) )$.
\end{lemma}
\begin{proof} 1] Consider a rational number $r$ such that $r > \max ( r(X-X') , r(Y-Y') )$. 
By definition, $X \leq X' + r E$ and $Y \leq Y' + r E$. Then, by Lemma \ref{lem:oplus}.1] we can write:
$$
X \oplus Y \leq \, ( X' + r E) \oplus Y \leq \, ( X' + r E) \oplus (Y' + r E) \, = \,(X' \oplus Y') + r E\,.
$$ Similarly we obtain: $(X' \oplus Y') \leq (X \oplus Y ) + r E$. The result is proved. 

2] Just apply part 1] with $X'=Y'=0$ and use $0\oplus 0=0$.
\end{proof}

\s 
We shall need to consider also a distance $(X,Y) \mapsto \| X-Y\|$ which is a bit more general than 
the one defined by $r(X-Y)$. 

\begin{definition} \label{def:norm}

We shall call F-norm on $\cF$, a map 
$\|  \,\| : \cF  \rightarrow \R^+$ satisfying the following properties for valid all $(t,X, Y ) \in \Q \times 
\cF \times \cF$.

\smallskip
\item 1] $ \| t X \|= | t | \| X\|$, $\;\| X + Y \| \leq \| X \| + \| Y\|$, $\;r(X) \leq \| X \|$.

\s
\item 2] For any $X,X',Y,Y' \in \cF$:
$$
\| X \oplus Y - X' \oplus Y' \| \, \leq \, \max ( \|X-X'\| , \|Y-Y'\| )\,.
$$

\end{definition}

 Assumption \ref{ass:2} and the inequality $r(X) \leq \|X \|$ show that: 
  $\|X \|=0\, \Rightarrow \,X=0$.
Notice also  that by Lemmas \ref{lem:dist} and \ref{lem:cont}, $X \mapsto r(X)$ defines 
a $F-$norm in the sense of Definition \ref{def:norm}. 


\s
\begin{definition} \label{def:dist} Let $\cF$ be a perfect semi-field as in Definition \ref{def:1} endowed with a F-norm $\| \, \|$ 
in the sense of Definition \ref{def:norm}.
One says that $\cF$ is complete 
if it is complete for the distance $(X,Y) \mapsto \| X-Y\|$. 
One says that $\cF$ is Banach if it is complete for the distance 
$(X,Y)\mapsto r(X-Y)$.

\end{definition}

In Proposition \ref{prop:w} we shall see that the class of   semi-fields which are complete for a $F-$norm
is stable by passing to the quotient by a closed congruence. 
A priori, it does not seem possible to prove directly  a similar result   for the  class of Banach semi-fields. Nevertheless,  see Section 5.

Now, let us give an explicit example of a $F-$norm which 
is not equal to $ X \mapsto r(X)$.

\s
\begin{example} Denote by $\cF_0$  the set of all the Lipschitz functions $X: [0,1] \rightarrow \R$ 
such that $X(0)=0$. Then $(\cF_0 , \max( , ), +)$ is a semi-field satisfying Assumptions \ref{ass:1} and \ref{ass:2} 
with $E(t)=t$. One obtains a $F-$norm on $\cF_0$ by setting 
$$
\| X \| = \sup_{t,t' \in [0,1],\, t >t'}\, \frac{ | X(t) - X(t')| }{t-t'}\, .
$$ One has $r(X) = \displaystyle \sup_{t \in ]0,1]}\, \frac{ | X(t)| }{t}$. Consider the element $X_0$ of $\cF_0$ defined by  $t \mapsto X_0(t)= e^t -1$. Then an easy computation shows that 
$r(X_0) < \| X_0 \|$. Actually, $(\cF_0,\| \, \|) $ is complete, whereas the completion of $\cF_0$ for 
$X \mapsto r(X)$ is the real vector space of all the continuous functions on the Stone-Cech compactification of $]0, 1]$.
\end{example}

\s
The next Lemma shows that the completeness hypothesis allows to extend continuously  the action of $\Q^*$ on $\cF$ to one of  $\R^*$. 

\begin{lemma} \label{lem:ext} Let  $\cF$ be a semi-field which  is complete for a F-norm $\| \, \|$.

\item 1] The maps $\theta_r, r\in \Q^{+*}$ of Lemma \ref{lem:Frob} extend by continuity to an action 
$\theta_t, t \in \R^{+*}$ on $(\cF, \oplus, +)$. For each $(t,X) \in \R^{+*}\times \cF$,
set  $ t X= \theta_t(X)$  and $(-t) X = -\theta_{-t}(X)$. Then, 
endowed with this action and with $\theta_0$, $(\cF, +)$ becomes a $\R-$vector space.

\item 2] For any $(t,X) \in \R \times \cF $, one has $\| t X\|= | t | \| X\|$.

\item 3] For any $(t, X, Y) \in \R^+ \times \cF \times \cF $:
$$
X \leq Y \Rightarrow t X \leq t Y\,.
$$
\item 4] Let $(t,t', X)\in \R \times \R \times \cF$ be such that $0 \leq X$. Then:
$$
t \leq t' \, \Rightarrow \,  t X \, \leq t' X \,.
$$
\end{lemma}
\begin{proof} 1] Let $t \in \R^{+*}$ and $X \in \cF$. Consider two sequence of 
 rational numbers $(r_n), (r'_n)$ converging to $t$. By Lemma \ref{lem:Frob}.2] and Definition \ref{def:norm}.1], one has 
 for all $n,p \in \N$:
 $$
 \| r_n X - r'_n X \| = | r_n - r'_n| \| X\|, \; \| r_n X - r_{n+p} X \| = | r_n - r_{n+p}| \| X\| \,.
 $$
 It is then clear that 
 $(r_n X)$ and  $(r'_n X )$ define two Cauchy sequences in $\cF$ which converge in $\cF$ 
 to the same limit; we denote it $ t X$.
 
 Now, using Lemma \ref{lem:Frob}.1] and Definition \ref{def:norm}.2] (which states the continuity of $\oplus$),   one obtains  
 for any $X, Y \in \cF$:
 $$t (X \oplus Y) =  \lim_{n \rightarrow + \infty} r_n (X \oplus Y) = \lim_{n \rightarrow + \infty} (r_n X \oplus r_n Y) = t X \oplus t Y\,.
 $$
This shows that $ X \mapsto t X$ defines an automorphism of $(\cF, \oplus, +)$.

 Next proceeding as in the proof 
 of Lemma \ref{lem:Frob}.2], one demonstrates that $(\cF, +)$ becomes endowed with a structure of  $\R-$vector space. 
 The result is proved.
 
 Lastly, 2], 3] and 4] are easy consequences of Lemmas \ref{lem:dist}  and \ref{lem:cont}, and are left to the reader.
\end{proof}

\s
The  Lemmas \ref{lem:dist} and \ref{lem:ext} show that $X \mapsto \| X\|$ defines a  norm, in the usual sense, on  $\cF\, (\not= \{0\})$,  giving it
 the structure of a 
real Banach vector space. Let us give a concrete example.

\begin{example} \label{ex:K} Let $K$ be a compact topological space. Then $(C^0(K ; \R), {\max}( ; ) , +,)$ defines a commutative Banach perfect semi-field,   with $E$ being the constant function {\bf 1} equal   to $1$. Here the F-norm is given by 
 $r(X)= \sup_{t\in K} | X(t) |$, for any $X\in  C^0(K ; \R)$.
\end{example}



\s
The next Lemma suggests that we have a sort of spectral theory for the elements $X$ of $\cF$ where 
$r(X)$ plays the role of a spectral radius satisfying $r(X) \leq \| X \|$. A  complete semi-field $(\cF , \oplus, +)$,   for a $F-$norm, 
 is somewhat analogous to a commutative Banach algebra, whereas a Banach semi-field is 
  analogous to a commutative $C^*-$algebra.

\begin{lemma} \label{lem:spec} Let  $\cF $ be a complete semi-field for a $F-$norm. Then the following are true:
\item 1]  $r(E)=1$.
\item 2] For any $X \in \cF$, $-r(X) E \leq X \leq r(X) E$.
\item 3] Let $X \in \cF$. Then $r(X) = \max ( r ( 0 \oplus X) , r ( 0 \oplus -X))$.

\end{lemma}
\begin{proof} 1] Recall that $0 \leq E$ so that, by Lemma \ref{lem:ext}.3],  $0 \leq t' E$ for any $t' \in \R^+$.
Since $1. E= E$, one has $r(E) \leq 1$. Assume, by contradiction, that $r(E) <1$. Then 
there exists $t \in ]0, 1[$ such that $ t\cdot E \oplus E= t\cdot E$. But, by distributivity of $+$:
$$
t\cdot E \oplus E= t\cdot E \oplus (t\cdot E + (1-t)\cdot  E) = t\cdot E +( 0 \oplus (1-t)\cdot  E)= (t+ (1-t)) E=E\,.
$$ Thus $t E= E$ why implies that $(1-t) E =0$. Since $\cF$ is not zero, Assumptions 1 and 2 show that 
$E$ cannot be zero. Therefore one gets a contradiction. This proves that $r(E)=1$. 

2] is a direct consequence of the continuity of the scalar map $t \mapsto t E$, of the Definition of $r(X)$ and of 
Definition \ref{def:norm}.2] which states the continuity of $\oplus$.

3] By lemma \ref{lem:oplus}.1],  the inequality $ X \leq r(X) E$ implies (given that $0 \leq E$):
$$
0 \oplus X \leq 0 \oplus r(X) E = r(X) E\,.
$$ Therefore, $r(0 \oplus X) \leq r(X)$ and similarly, by Lemma \ref{lem:dist}.3],  $r(0 \oplus -X) \leq r(-X) = r(X)$.
Now let us prove the reverse inequality. By Lemma \ref{lem:oplus}.3]:
$$
X \leq 0 \oplus X \leq r(0 \oplus X) E\,.
$$ Similarly, $- X \leq r(0 \oplus -X) E $. Therefore, $r(X) \leq \max ( r(0 \oplus X) , r(0 \oplus -X) )$. The result is proved.
\end{proof}

\s
Let us explain the content of the previous Lemma on a concrete example.
\begin{example} We use the notations of Example \ref{ex:K}. Let $X \in C^0(K , \R)$. Then 
$$r( 0 \oplus X) = \max_{v \in K}\, X(v)\,.$$ Moreover, for any real $\epsilon \geq 0$, 
$$
\epsilon {\bf 1} \leq r( 0 \oplus X)\, {\bf 1} - (0 \oplus X) \, \; \Rightarrow \,\; \epsilon= 0\,.
$$
\end{example}
\m



\subsection{ Spectrum of a commutative  perfect semi-field $\cF$ of characteristic $1$.} $\;$

\m
The following definition of character is somewhat analogous to  the concept of character of a commutative complex Banach algebra.

\begin{definition} \label{def:char} A character of $\cF$ is a map $\phi: \cF \rightarrow \R$, not identically zero, satisfying the following properties valid for any 
$(X,Y) \in \cF \times \cF$:
\item 1] $ \phi (X \oplus Y) = {\max}\, ( \phi(X) , \phi(Y) )$.

\item 2] $\phi (X +Y)= \phi(X) + \phi(Y)$.


The set of characters satisfying the normalization condition $\phi(E)=1$ is called the spectrum $S_E( \cF)$ of $\cF$.
\end{definition}

\s
\begin{remark} Using Assumption 1 and the fact that $E \geq 0$, one observes that any character $\phi$ satisfies $\phi (E) >0$. Then, morally a character is a geometric point of "Spec $\cF$". The Frobenius action 
on the set of characters $\phi $ is given by $(\lambda \cdot \phi)(X)= \lambda \,\phi (X)$, $\lambda \in \R^{+*}$. 
An element of $S_E( \cF)$ defines morally a closed point:  an orbit of geometric points under the action 
of $\R^{+*}$.
\end{remark}

\smallskip 
Let us give an explicit  example of a character.
\begin{example} With the notations of Example \ref{ex:K}. Every point $x \in K$ defines an element $\phi_x$
of $S_E( C^0(K , \R))$ by $\phi_x:  f \mapsto f(x)$.
\end{example}
\m
Here are  several elementary properties of the normalized characters.

\begin{lemma} \label{lem:char} Let $\cF$ be a semi-field as in Definition \ref{def:1}.
 Consider $\phi \in S_E( \cF)$ and 
 
 \noindent $(t, X,Y ) \in \R \times \cF \times \cF$ 
such that $X \leq Y$. Then the following are true:
\item 1]   $\phi (X) \leq \phi (Y)$ and $| \phi(X) | \leq r(X)$. 
 \item 2] Assume that $\cF$ is a complete semi-field  for a $F-$norm. Then,  $\phi (t  X) = t \, \phi(X)$.
\end{lemma}
\begin{proof} 1] Since $X \oplus Y =Y$, 
the inequality  $\phi (X) \leq \phi (Y)$ is a consequence of Definition \ref{def:char}.1].  Next, by definition of $r(X)$, there exists a 
sequence of rational numbers $(r_n)$ such that $\lim_{n \rightarrow + \infty} r_n= r(X)$ and:
$$
\forall n \in \N,\; r(X) \leq r_n,\; -r_n E\leq X \leq r_n E \,.
$$ From Definition \ref{def:char}.2],  one obtains $\phi (r_n E)= r_n \phi(E)= r_n$ and similarly, 
 $\phi (-r_n E)=- r_n$. Then, we deduce from all of this that $\forall n \in \N,\; -r_n \leq \phi (X) \leq r_n$. 
 By letting $n \rightarrow + \infty$, one obtains the desired result.

2] 
Now set $X_+ = 0 \oplus X, X_- = 0 \oplus - X$ so that 
by Lemma \ref{lem:oplus}, $X = X_+ - X_-$. Then consider the map $\psi : \R \rightarrow \cF$ 
defined by $\psi (t) = \phi( t X_+)$. By definition of $\phi$, $\psi ( t+t')= \psi(t) + \psi(t')$ for any reals $t,t'$. 
By Lemma \ref{lem:ext}.4], $\psi$ is nondecreasing on $\R$. Then it is well known that for any real $t$,
$\psi (t)= t \psi(1)$, in other words: $t \phi(X_+) = \phi( t X_+)$ and similarly $t \phi(X_-) = \phi( t X_-)$.
Since $\phi (-Z)=-\phi(Z)$ for any $Z \in \cF,$
one then obtains easily that $ \forall t \in \R,\; \phi (t  X) = t \, \phi(X)$.

\end{proof}

\m
Now, motivated by the theory of commutative complex Banach algebras (\cite{Ru}), we 
introduce the following topology $\cT $ on $S_E( \cF)$. 
\begin{definition} \label{def:topo} Let $\cF$ be as in Definition \ref{def:1}.
We denote by $\cT$ the weakest topology on $S_E( \cF)$ rendering continuous all the maps 
$\phi \mapsto \phi(X)$ where $X$ runs over $\cF$. 
Let $\phi_0 \in S_E( \cF)$, then a system of fundamental open neighborhoods of $\phi_0$ for $\cT$ is given by 
the open subsets:
$$
V_{\phi_0} (\epsilon, X_1,  \ldots , X_n) = \{ \phi \in S_E( \cF)\,/\,  | \phi(X_j) -\phi_0(X_j) | < \epsilon,\, 1\leq j \leq n \}\,,
$$ where 
$\epsilon >0$ and the $X_1,  \ldots , X_n $ run over $\cF$.
\end{definition}

\s The next result is analogous to the compacity of the set of characters of a commutative Banach algebra 
(\cite{Ru}). 
\begin{theorem} \label{thm:comp} $S_E(\cF)$ is a compact topological space for the previous topology $\cT$.
\end{theorem}
\begin{proof} Consider the following product (or functions space):
$$
\F = \prod_{X \in \cF}\, [-r(X) , r(X)]\,.
$$ We endow it with the product topology associated to the  compact intervals $[-r(X) , r(X)]$.
By Tychonoff's theorem, $\F$ is compact. The open subsets of this product topology are given 
by the subsets:
$$
\prod_{j=1}^n I_j \times \prod_{X \in \cF \setminus \{X_1, \ldots , X_n\}} \, [-r(X) , r(X)]\, ,
$$ where the $X_j$ ($1 \leq j\leq n$) run over $\cF$ and each  $I_j$ is an open subset of $[-r(X_j) , r(X_j)]$.

 Lemma \ref{lem:char}.1] shows that $S_E(\cF)$ is naturally included in $\F$. 
Then, we  need two intermediate lemmas.
\begin{lemma}The product topology of $\F$ induces the topology $\cT$ of $S_E(\cF)$.
\end{lemma} 
\begin{proof} With the notations of Definition \ref{def:topo}, $ V_{\phi_0} (\epsilon, X_1,  \ldots , X_n)$ 
is nothing but the intersection of $S_E(\cF)$ (viewed as a subset of $\F$) with the subset
$$
\prod_{j=1}^n I_j \times \prod_{X \in \cF \setminus \{X_1, \ldots , X_n\}} \, [-r(X) , r(X)]\, ,
$$ where each $I_j$ is equal to $[-r(X_j), r(X_j)] \cap ]-\epsilon + \phi_0(X_j),  \epsilon + \phi_0(X_j)[$, ($1\leq j \leq n$).
Therefore, the $ V_{\phi_0} (\epsilon, X_1,  \ldots , X_n)$ constitute a basis of open neighborhoods of $\phi_0$ for the 
topology of $S_E(\cF)$ induced by the (product) one of $\F$. The Lemma is proved.
\end{proof}

Now the Theorem follows from the next Lemma.
\begin{lemma} $S_E(\cF)$ is closed in the compact space $\F$.
\end{lemma} 
\begin{proof} Let $\psi \in \overline{S_E(\cF)}$ and $X,Y \in \cF$. For any $\phi \in S_E(\cF)$  one can write the following inequalities:
\begin{equation} \label{eq:a}
| \psi(E) -1 | \leq | \psi(E)-\phi(E)|
\end{equation}
\begin{equation} \label{eq:b}
| \psi(X+Y) - \psi(X) - \psi(Y)| \leq | \psi( X+Y) - \phi(X+Y) | + | \psi(X) - \phi(X) | + | \psi(Y) - \phi(Y) |
\end{equation}
\begin{equation} \label{eq:cc}
| \psi( X \oplus Y) - \max ( \psi(X) , \psi(Y) ) | \leq 
| \psi( X \oplus Y) - \phi( X \oplus Y) |  + | \psi(X) - \phi(X) | + | \psi(Y) - \phi(Y) |
\end{equation}
Notice that  \eqref{eq:cc} is a consequence of the following inequality.
$$
|  \max ( \psi(X) , \psi(Y) ) \,-\, \max ( \phi(X) , \phi(Y)  | \leq | \psi(X) - \phi(X) | + | \psi(Y) - \phi(Y) |\,.
$$ 
Since  $\psi \in \overline{S_E(\cF)}$,  for any $\epsilon >0$, one can find $\phi \in S_E(\cF)$ such that 
the right hand sides of \eqref{eq:a}, \eqref{eq:b} and \eqref{eq:cc} are all lower than $\epsilon$.
This shows  that $\psi \in S_E(\cF)$. Thus the Lemma is proved.
\end{proof}
\end{proof}

Notice that we did not assume $\cF$ to be complete in the previous Theorem. This observation will be used 
in Sections 6 and 7 for algebraic goals.

\begin{remark} We shall prove later that $S_E(\cF)$ is not empty; this is not obvious.  Besides, it will be clear in Section 4 that  if $\cF$ is not separable then
the topological space $(S_E(\cF) , \cT) $ is not metrizable.
\end{remark}

\m
\section{Congruences on $(\cF, \oplus, +)$.} $\;$

\s

\subsection{Algebraic properties of congruences and operations on quotient spaces.} $\;$

\s

We shall use the notion of congruence (\cite{Iz}, \cite{Lescot}, \cite{Pareigis}) as the analogue of the notion of 
ideal in the theory of complex Banach algebras. Of course, $\cF$ will denote a semi-field as in Definition \ref{def:1}.

\begin{definition} \label{def:congF} A congruence $\sim$ on $\cF$ is an equivalence relation on $\cF$
satisfying the following conditions valid for all $X,Y, Z \in \cF$: 

\item  If $X \sim Y$ then $-X \sim -Y$,  $X+Z \sim Y+Z$, and $X\oplus Z \sim Y\oplus Z$.

We shall denote by $\pi: \cF \rightarrow \cF/\sim$ the natural projection onto the set of equivalence classes.
The trivial  congruence is the one such that
for all $X,Y \in \cF,$ $X {\sim} Y$.
\end{definition}

\s The next Lemma allows to understand better this concept, its part 3] reveals a key absorption phenomenon.
\begin{lemma} \label{lem:cong} \item 1] Let $\sim$ be a congruence on $\cF,$ then the class of $0$, denoted $\pi(0)$, is a sub semi-field of $(\cF, \oplus, +)$ such that the following two properties hold:
$$
\forall X, Y \in \cF, \; \; X \sim Y\, \Leftrightarrow  \, X-Y \in \pi(0) \,,
$$
 and
 \begin{equation} \label{eq:s}
\forall (X, Y, Z) \in \cF \times \cF \times \pi(0) , \; \exists Z_1 \in \pi(0),\; (X+Z) \oplus Y= (X\oplus Y) + Z_1 \,.
\end{equation}  
\item 2] Conversely, let $\cH$ be a sub semi-field of $(\cF, \oplus, +)$ satisfying \eqref{eq:s} with $\pi(0)$ replaced by 
$\cH$.  Then one defines a congruence on $\cF$ by setting: $X \sim Y\, \Leftrightarrow  \, X-Y \in \cH$. 
\item 3] Let $\sim$ be a congruence of $\cF$. Consider $A,C \in \pi(0)$ and $B \in \cF $ such that $A\leq B\leq C$.
Then  $B $  belongs  to $\pi(0)$ too.
\end{lemma} 
\begin{proof} 1] Let $X, Y \in \pi(0),$ 
 since $
0 \sim X$, one has by definition of a congruence: $ 0 \oplus Y \sim X \oplus Y$. Similarly, 
$0 \sim Y$ implies that $0= 0 \oplus 0 \sim 0 \oplus Y$. Therefore by transitivity, 
$X \oplus Y \in \pi(0)$. In the same way one obtains that $(\pi (0), +)$ is an additive sub group of $\cF$.
Next, by the very definition of a congruence:  $X \sim Y$ iff $0 \sim Y-X$. 
 We leave the rest of the proof to the reader.

 2] is easy and also left to the reader. Let us prove 3]. Consider $A,B,C $ as in the statement of 3].
By hypothesis one has $A \oplus B=B $ and $ C\oplus B = C$.
Then, by definition of a congruence one has:
\begin{equation} \label{eq:A}
 0 \sim A \, \Rightarrow \, 0 \oplus B \sim A \oplus B = B ,\, 
\end{equation}

\begin{equation} \label{eq:B}
C \sim 0 \, \Rightarrow \, C \oplus B \sim 0 \oplus B\,.
\end{equation}
 Now,  since $C \oplus B = C \sim 0$, \eqref{eq:B} implies that $ 0 \sim 0 \oplus B$. 
 But, by \eqref{eq:A},  $0 \oplus B \sim B$. By the transitivity of $\sim$, one obtains 
 $B \sim 0$,  which proves the result.
 
\end{proof}

\s
The  following  lemma describes  the simplest example, but it is already instructive.
\begin{lemma} Let $\sim$ be a congruence on $(\R, \max( ,), +)$. Then either $X \sim Y$ for all 
$X,Y \in \R$ or,  one has: $X \sim Y \, \Leftrightarrow X=Y$ for any $X,Y \in \R$.
\end{lemma}
\begin{proof} Assume that there exist two distinct reals $X,Y$ such $X \sim Y$.
By Lemma \ref{lem:cong}.1], this means that 
$\pi(0)$ is a non trivial subgroup of $\R$, so it contains a positive real $X >0$. By definition of a congruence, 
for any $n \in \N$, $-n X \sim 0$ and $n X \sim 0$.   But by Lemma \ref{lem:cong}.3] we obtain 
that $[-n X , n X] \subset \pi(0)$. Therefore $\sim$ is the trivial congruence, which proves the result. 
\end{proof}
\s
Here is a more elaborate example.
\begin{example} \label{ex:K_1} With the notations of Example \ref{ex:K}, consider a compact subset $K_1$ of $K$.
Then one defines a (closed) congruence   $\sim_{K_1}$ on $C^0(K , \R)$ by setting: $f \sim_{K_1} g$ if 
$f_{| _{K_1}}= g_{|_{K_1}}$.
\end{example}

\s
\begin{lemma} \label{prop:rel} Assume that $\cF$ is complete for a $F-$norm (see Definition \ref{def:dist})  and 
 let $\sim$ be a congruence on $\cF$. View $\sim$ as a subset of $\cF \times \cF$ 
and denote by $\, \overline{\sim}$ its closure in $\cF \times \cF$. Then $\overline{\sim}$ defines also
a congruence on $\cF$, and 
$$
\forall X, Y \in \cF , \; X \, \overline{\sim} \,Y\, \Leftrightarrow X-Y \in \overline{\pi(0)} ,
$$ where $\overline{\pi(0)}$ denotes the closure of  $\pi(0)$. Therefore, $\sim$ is closed if and only if $\pi(0)$ is a closed subset of
$\cF$.
\end{lemma}
\begin{proof}
By definition,  $X \,\overline{\sim} \,Y$ if and only one can find two sequences $(X_n)_{n \in \N},\, (Y_n)_{n \in \N}$ of 
points of $\cF$ such that:
$$
\lim_{n\rightarrow +\infty} X_n = X,\; \lim_{n\rightarrow +\infty} Y_n =Y, \; {\rm and} \; \forall n \in \N,\; X_n - Y_n \in \pi(0)\,.
$$ Therefore , $X \overline{\sim} Y$ implies that $X-Y \in \overline{\pi(0)}$. Conversely, set $H= X-Y \in \overline{\pi(0)}$ and 
consider a sequence $(H_n)$ in $\pi(0)$ such that $\lim H_n=H$. Set $X_n= Y + H_n$ and $Y_n=Y$.
By construction,  $X_n \sim Y_n$, $\lim X_n = Y+H\,=\,X$ and $\lim Y_n=Y$. 
We then conclude that 
$X \,\overline{\sim} \, Y$ if and only $X- Y$ belongs to  $ \overline{\pi(0)}$. Since $(\overline{\pi(0)} , +)$ is a subgroup of $(\cF, +)$,
  it becomes then clear that $\overline{\sim}$
is  an equivalence relation.

Consider now $X,Y,Z \in \cF$ such that $ X \, \overline{\sim} \,Y$, there exists 
two sequences $(X_n)_{n \in \N},\, (Y_n)_{n \in \N}$ of $\cF$ such that
$$
\lim_n X_n=X,\; \lim_n Y_n =Y,\; {\rm and}\; \forall n \in \N,\; X_n \sim Y_n\,.
$$ By definition of a congruence, one has $-X_n \sim -Y_n$, $X_n \oplus Z \sim Y_n \oplus Z$  and $X_n + Z \sim Y_n + Z$ for 
each $n \geq 0$. By taking the limits and using Definition \ref{def:norm}.2],  one gets 
$$ -X\, \overline{\sim} \,-Y, \; X \oplus Z\, \,\overline{\sim} \, \,Y \oplus Z , \;
X + Z \,\,\overline{\sim} \, \,Y + Z\,. 
$$ The Lemma is proved.
\end{proof}

\s
By considering the structure of semi-field  induced on the set of equivalence classes $\cF/\sim$ by the one of $\cF$, 
one gets a better understanding of the congruence $\sim$.

\begin{proposition} \label{prop:alg} Let $\sim$ be a congruence on $\cF$.
\item 1] The laws $\oplus, +$ of $\cF$ induce 
two laws  (denoted by the same symbols) on $\cF/\sim$ which gives it a natural structure of  perfect semi-field 
 in the sense of Definition \ref{def:1}. Moreover,  the projection $\pi: \cF \rightarrow \cF/\sim$ is a homomorphism of semi-fields.

\item 2] For any $(t,X,Y ) \in \Q \times \cF \times \cF $ one has:
$$
X \sim Y \; \Rightarrow t  X \sim t Y , \, {\rm and}
$$
the class of $0$, $\pi(0)$,  is a $\Q-$subvector space of $\cF$. 
\item 3] Let $X, Y \in \cF$. Then $\pi(X) \leq \pi(Y)$  if and only if 
there exists $ H \in \pi(0)$ such that $X \leq (Y + H)$.

\end{proposition}
\begin{proof} 1] is easy and left to the reader. 

2] Since $(\cF/\sim, \oplus, +) $ satisfies the conditions of Definition \ref{def:1},  we can apply to it 
Lemma \ref{lem:Frob}, which gives immediately the desired result.

3] Assume that $\pi(X) \leq \pi(Y)$.  Then $X \oplus Y \sim Y$. So there exists $ H \in \pi(0)$, 
such that 

\noindent $X \oplus Y = Y+H$. 
Since $X \oplus Y\oplus Y=X \oplus Y$, this implies:
$$
Y+H= X \oplus Y \oplus Y=   (Y+H) \oplus (Y+0) = Y + ( H \oplus 0).
$$ Thus $H= H \oplus 0$ so that  $H \geq 0$. This implies:  $Y \oplus (Y+H)=Y+H$. We then obtain:
$$
Y+H = (Y+ H) \oplus (Y+H)= X \oplus Y \oplus (Y +H)= X \oplus (Y+H)\,.
$$ Therefore, $X \leq (Y+H)$ which proves the first implication $\Rightarrow$.  The converse is trivial.
\end{proof}

\s
Now let us examine the quotient semi-field $\cF/\sim$ when $\cF$ is complete and $\sim$ is closed.

\begin{proposition} \label{prop:w} Assume that  $\cF$ is  a complete semi-field for a $F-$norm $\| \, \|$ and let
 $\sim$ be a closed congruence on $\cF$ (i.e it induces  a closed subset of $\cF \times \cF$).
Then: 
\item 1]  The   semi-field $\cF/\sim$ satisfies Assumptions 1 and 2 with $\pi(E)$ in place of $E$. Denote by 
$r_\sim (\pi(X))$ (instead of $r(\pi(X)$) the corresponding real number  associated to $\pi(X) \in \cF/\sim$ . 
\item 2] The class $\pi(0)$ defines a closed real sub-vector space of $\cF$ and
for any $(t, X,Y ) \in\R \times  \cF \times \cF $ one has:
$$
X \sim Y \; \Rightarrow t X \sim t Y\,.
$$ In any words, the projection $\pi: \cF \rightarrow \cF/\sim\, = \cF / \pi(0)$ defines a  $\R-$linear map 
between two $\R-$vector spaces. 
\item 3] For any $X \in \cF$, one has:
$$
   r_\sim (\pi(X) )\,  \leq \, \| \pi(X) \|_1= \inf_{Z \in \pi(0)}\, \|X +Z\|\, ,
$$ and the real vector space $\cF/\sim$ is complete for the above norm $\| \, \|_1$.
\item 4] For any $X_1, X'_1, Y_1, Y'_1 \in \cF/\sim$, one has:
$$
\| X_1 \oplus Y_1 - X'_1 \oplus Y'_1 \|_1 \leq \max ( \| X_1- X'_1\|_1 , \| Y_1- Y'_1\|_1) \,.
$$

In other words, $\| \, \|_1$ defines a $F-$norm on $\cF/\sim$ . 

\end{proposition}
\begin{proof} 1] Consider an element of $\cF/\sim$, it is of the form $\pi(X)$ for a suitable $X \in \cF$.
By Assumption 1 for $\cF$, there exists $t\in \Q^+$ such that 
$ - t E \leq X \leq t E$. This implies:
$$
-t \pi(E) \leq \pi(X) \leq t \pi(E)\,.
$$ Therefore, $(\cF/\sim , \pi(E))$ satisfies Assumption 1.  Assume now  that 
$r(\pi(X))=0$.  By Lemma \ref{lem:dist}.1] this means that: 
$$\forall t\in \Q^+,\;
 -t \pi(E) \leq \pi(X) \leq t \pi(E) \,.$$ This implies, by taking $t=1/n$,  that for 
 any $n\in \N^*$, there exist  $H_n, H'_n \in \pi(0)$ such that:
 $$
 \frac{-1}{n} E \oplus X= X + H_n, \;  \;X \oplus \frac{1}{n} E = \frac{1}{n} E + H'_n  \,.
  $$ Letting $n \rightarrow + \infty$ and using Definition \ref{def:norm}.2],  one obtains: 
  $$
  0 \oplus X = X + \lim_n H_n , \;  \;X \oplus 0 = 0 + \lim_n H'_n \,.
  $$ But by Lemma \ref{prop:rel}, $\pi(0)$ is closed. We then deduce that
  $0 \oplus X \sim X, \;  X \oplus 0 \sim 0$.  This  implies that $\pi(X)=0$ and
   $\cF/\sim$ satisfies the Assumption 2. 
   
   2] By Lemma \ref{lem:ext}.1] and by hypothesis, $\cF$ is a real vector space which is complete for $\| \, \|$.  
   By the previous Proposition, $\pi(0)$ is a $\Q-$sub vector space of $\cF$. Since 
   $\pi(0)$  is closed, it is also a real sub vector space. Then, using  Lemma \ref{lem:cong}.1],  one obtains easily the desired result.
   
   3] One checks easily that for all $X \in \cF$:
   $$
   r_\sim (\pi(X) )\,  \leq \,  \inf_{Z \in \pi(0)}\, r(X +Z)\, .
   $$ But, by Definition \ref{def:norm}.1], $r(X+ Z) \leq \| X+Z \|$, so that the required inequality follows  immediately. Lastly,  the fact that $(\cF/\sim , \| \, \|_1)$ is complete is an easy consequence of the completeness of $(\cF, \| \, \|)$.
   
   4] Let  $X,X',Y,Y' \in \cF$ be such that $\pi(X)=X_1,\, \pi(X')=X'_1,\, \pi(Y)=Y_1,\, \pi(Y')=Y'_1$. 
   Consider also $H \in \pi(0)$. By Definition \ref{def:norm}.2] one has:
   $$
   \| H + \, X \oplus Y - X' \oplus Y' \|= \| (H + X) \oplus (H+Y) - X' \oplus Y' \| \leq
   $$
   $$
   \max ( \| H+X-X'\| ,\|H+Y-Y'\| )\,.
   $$ Therefore:
   $$
   \| X_1 \oplus Y_1 - X'_1 \oplus Y_1' \|_1 \leq \max ( \| H+X-X'\| ,\|H+Y-Y'\| )\,.
   $$ Since this inequality holds for all $X, Y$ satisfying $\pi(X)=X_1, \pi(Y)=Y_1$,  we deduce:
   $$
   \| X_1 \oplus Y_1 - X'_1 \oplus Y_1' \|_1 \leq \inf_{A,B \in \pi(0), }\,  \max ( \| A+X-X'\| ,\|B+Y-Y'\|) \,.
   $$
   One checks easily that the right hand side is $\leq \max( \| X_1 -X'_1\|_1 , \| Y_1 -Y'_1\|_1)$. 
   The result is proved.
  \end{proof}

\s  Proposition \ref{prop:w}.3] suggests that it is not clear at all whether, or not, 
$\cF/\sim$ is complete for the norm $r_\sim (X_1)$. This is precisely the reason why 
we introduced the concept of $F-$norm in Definition \ref{def:norm}.
Let us explain the content of  this  on a concrete example.
\begin{example} We use the notations of Examples \ref{ex:K} and \ref{ex:K_1}. 
So $\cF = C^0(K, \R)$ is a Banach semi-field for the norm $r(X)= \max_{t \in K} | X(t) | $. 
Consider the congruence $\sim_{K_1}$ associated to a compact subset $K_1$ of $K$.
  Then,  $X \sim_{K_1} Y$ means that the continuous function $X-Y$ vanishes on $K_1$.
 Moreover,   $\pi(0)$ is exactly the set of continuous functions which vanish on $K_1$.
 Since a compact space is normal, Urysohn's extension Theorem \cite{Sc}[Page 347]  implies the following  natural identification: 
$$\cF/ \sim_{K_1} \, = \, C^0(K_1, \R)\, , 
$$ where the class $\pi(X) $ of the function $X$ is identified with $ X_{_{|K_1}}$.
 Now, the inequality of  Proposition \ref{prop:w}.3] means that:
$$
 \sup_{w \in K_1} | X_{|_{K_1}} (w) |\, \leq \, \inf_{ Z \in \pi(0)}\, \sup_{v \in K} | X(v) + Z(v) | \,.
$$ 
Actually, Urysohn's theorem implies that this inequality is in fact an equality. We shall come back to this in Section 5. 

\end{example}


\s

\subsection{Maximal Congruences. An analogue of Gelfand-Mazur's Theorem.} $\;$

\m 
In the rest of this Section we fix a complete semi-field $\cF$ for a $F-$norm $\| \, \|$ as in Definition \ref{def:dist}.

We recall below the natural partial order between congruences of $\cF$. It is analogous to the inclusion between 
ideals in a ring $A$, the trivial congruence of $\cF$ has only one class and is the analogue of $A$.
\begin{definition} \label{def:order} One defines a partial order $\leq$ on the set of congruences of $\cF$ 
in the following way. One says that $\sim_1 \,\leq \,\sim_2$ if: $$ \forall X,Y \in \cF,\; X \sim_1 Y \Rightarrow X \sim_2 Y\,.$$
Denote by $\pi_j(0)$ the class of the congruence $\sim_j, \, j=1,2$. 
Then $\sim_1\, \leq\, \sim_2$ if and only if $\pi_1(0) \subset \pi_2(0)$.
Next, one says that a congruence $\sim$ is maximal if  $\sim$ is not trivial and if $\sim \, \leq \widehat{\sim}$ implies either 
that $\sim = \widehat{\sim}$ or, that $\widehat{\sim}$ is the trivial congruence.

\end{definition}
\s
The following Theorem is the analogue of the fact that in a commutative complex Banach algebra, 
any ideal is included in a maximal ideal and that any maximal ideal  is closed.
\begin{theorem} \label{thm:max}
\item 1] Let $\sim_1$ be a non trivial congruence of $\cF$. 
Then there exists a maximal congruence $\sim_2$ on $\cF$ such that 
$\sim_1 \,\leq \,\sim_2$.
\item 2] Each maximal congruence $\sim$  on $\cF$ is closed, i.e. it defines a closed subset of $\cF \times \cF$.
\end{theorem}
\begin{proof}
1] Denote by $\Lambda$ the set of non trivial congruences $\sim$ of $\cF$ such that $\sim_1 \leq \sim$.
Consider a chain $\cC \subset \Lambda$, this means that $\cC$ is totally ordered for the order $\leq$.
One then defines a congruence $\sim_u$ on $\cF$ by saying that $X \sim_uY$ if there exists 
$\sim \in \cC$ such that $X \sim Y$. Clearly, $\sim_u$ is an upper bound for $\cC$ in $\Lambda$. Then, 
by Zorn Lemma, $\Lambda$ admits a maximal element. This proves the result.

 2] Suppose, by the contrary, that $\sim$ is not closed. The  hypothesis that $\sim$ is 
maximal  and Lemma \ref{prop:rel}  imply  that 
$\overline{\sim}$ is the trivial congruence so that in particular  $ 0 \, \overline{\sim} \,E$. 
This means that $E$ belongs to the closure of $\pi (0)$, the class of $0$ for $\sim$.
Let $\epsilon \in ]0, 1/5[ \cap  \Q$, then we can find $Y \in  \pi(0) $ such that $ \| Y -E \| \leq \epsilon/2$. 
But, by Definition \ref{def:norm}.1], $r(Y-E) \leq \| Y -E \|$ so that we can write:
$$
 -\epsilon E \leq Y-E \leq \epsilon E\,.
$$ This implies that $0 \leq (1-\epsilon) E \leq Y$. But since $0 \sim Y$, Lemma \ref{lem:cong}.3] shows that 
$(1-\epsilon) E \sim 0$. Recall that $ 1 - \epsilon \in \Q^{+*}$, then, since $\pi(0)$ is a $\Q-$sub vector space of $\cF$, we obtain that 
$\forall t \in \Q, \, t  E \sim 0$. By Assumption 1 and Lemma \ref{lem:cong}.3],  we conclude
 that for all $X \in \cF$, $X \sim 0$. This is a contradiction.
 The Theorem is proved.
\end{proof}
\s
The following Theorem is analogous to the one of Gelfand-Mazur stating that a complex Banach algebra which is a field is isomorphic to $\C$. 
\begin{theorem} \label{thm:gm} Let $\sim$ be a maximal congruence on $\cF$.
 Then   $\cF/\sim\,  = \, \{ t \pi(E)/ \, t \in \R\} $,  and the map $t \pi(E) \mapsto t$
 defines an  isomorphism of  semi-fields from $(\cF, \oplus, +)$ onto
  $(\R, {\max}( ,), +)$.
\end{theorem}
\begin{proof} To simplify the notations set $E_1= \pi(E)$ and denote a general element of $\cF/\sim$ 
by $X_1$ (i.e. with the subscript $1$). 
Suppose, by the contrary, that there exists $X \in \cF$ such that $X_1= \pi(X)$ 
does not belong to $\R E_1$. By the previous Theorem $\sim$ is closed, so that Proposition \ref{prop:w}.4] allows to apply Lemma \ref{lem:spec}.3] to $\cF/\sim$. Then, at the expense of replacing $X_1$ by $-X_1$, we can assume 
that $r_\sim ( X_1 )$ is the smallest real $t$ such that $X_1 \leq t E_1$. In order to simplify the notation, set:
$$
\Xi_1= \bigl( r_\sim ( X_1 ) E_1 - X_1\bigr) \, \geq 0\,.
$$

Consider now the  subset  $\cH_1$ of $\cF/\sim $ defined by 
\begin{equation} \label{eq:cong}
\cH_1= \{ H_1 \in \cF/\sim \; , \; \exists \lambda_1 \in \R^+,\; -\lambda_1 \Xi_1 \leq H_1 \leq
 \lambda_1 \Xi_1 \} \,.
\end{equation} Using Proposition \ref{prop:w}.4] and Lemma \ref{lem:ext}, one checks easily that $\cH_1$ is a real sub vector space of $\cF/ \sim$. 
We are going to show that the relation $\sim_1$, defined by $A_1 \sim_1 B_1$ if $A_1-B_1 \in \cH_1$,  induces a non trivial congruence on $\cF/\sim$ by following 
Lemma \ref{lem:cong}.2].  

Consider $H_1, H_2 \in \cH_1$ satisfying the inequalities of \eqref{eq:cong} with respectively 
$\lambda_1, \lambda_2 \geq 0$. Then, using Lemma \ref{lem:oplus} it is easily seen that:
$$
-\max( \lambda_1, \lambda_2) \,\Xi_1 \leq H_1\oplus H_2 \leq \max( \lambda_1, \lambda_2)\, \Xi_1\,.
$$ Therefore, $H_1 \oplus H_2 \in \cH_1$. 

By Lemma \ref{lem:cong}.2], the next Lemma will imply that  the relation $\sim_1$ (associated above to $ \cH_1$),  defines a congruence 
on $\cF/\sim$.

\begin{lemma}
 Consider  $A_1, B_1 \in \cF/\sim$ and $H_1 \in \cH_1$ as in 
\eqref{eq:cong}. Then there exists $H_2 \in \cH_1$ such that
 $$(A_1 + H_1) \oplus B_1 = A_1 \oplus B_1 + H_2\,.$$ 
 \end{lemma} 
 \begin{proof} Using the inequality  $0 \leq \Xi_1$ and the fact that $H_1 \in \cH_1$ , we obtain: 
 $$ (A_1 +H_1)\leq A_1 + \lambda_1 \Xi_1, \;  B_1 \leq B_1 + \lambda_1 \Xi_1 \,.
 $$
 Using the commutativity (and associativity) of $\oplus$ and the definition of $\leq$, we then deduce:
 $$
 (A_1 + H_1) \oplus B_1 \oplus \bigl( (A_1 \oplus B_1) +\,   \lambda_1 \Xi_1 \bigr)= 
 (A_1 + H_1) \oplus B_1 \oplus \Bigl( (A_1 + \lambda_1 \Xi_1)  \oplus (B_1 + \lambda_1 \Xi_1)  \Bigr) \,=
 $$
 $$
 (A_1 + \lambda_1 \Xi_1)  \oplus (B_1 + \lambda_1 \Xi_1) = (A_1 \oplus B_1) +\,   \lambda_1 \Xi_1\,.
 $$ But this means precisely that:
 $$
 (A_1 + H_1) \oplus B_1 \leq (A_1 \oplus B_1) +\,   \lambda_1 \Xi_1\,.
 $$ Similarly, we can prove that:
 $
 (A_1 \oplus B_1) -\,   \lambda_1 \Xi_1\, \leq (A_1 + H_1) \oplus B_1\,.
 $ Therefore, we obtain:
 $$
 -\,   \lambda_1 \Xi_1\, \leq (A_1 + H_1) \oplus B_1 \, - \, (A_1 \oplus B_1) \leq \lambda_1 \Xi_1\,.
 $$ The Lemma is proved.
 \end{proof}
Now we check that $\sim_1$ is not trivial. Suppose, by the contrary, that $E_1 \in \cH_1$, then:
$$
\exists \lambda_1 \in \R^{+*},\; 0\leq  E_1 \leq  \lambda_1 \Bigl( r_\sim (X_1) E_1 -X_1 \Bigr)\,.
$$ 
This implies that $ \lambda_1 X_1 \leq (\lambda_1 \,r_\sim (X_1)-1)  E_1$. Therefore, by Lemma \ref{lem:ext} one obtains:
$$
X_1 \leq \frac{(\lambda_1\, r_\sim (X_1)-1)}{\lambda_1}  E_1\,.
$$ But this inequality contradicts the definition of $r_\sim (X_1)$ (see Proposition \ref{prop:w}.1]).
Therefore, $\sim_1$ is indeed not trivial, so the set of equivalence classes of $\sim_1$ in 
 $\cF/\sim$ defines  a (non trivial) semi-field $\cS_1$. 
  Denote by $\Pi_1: (\cF/\sim) \rightarrow \cS_1$ the associated projection. Consider then the (surjective) homomorphism 
 of semi-fields $ \Pi_1 \circ \pi: (\cF, \oplus, +)  \rightarrow (\cS_1, \oplus, +)$.  One defines a congruence $\widehat{\sim}$ on $\cF$ 
  by saying that
 $X \,\widehat{\sim}\, Y$ if $\Pi_1 \circ \pi(X)=\Pi_1 \circ \pi(Y)$. 
 
 By hypothesis $\Xi_1$ is not zero and by construction $\Pi_1(\Xi_1)=0$. This means that 
 
 \noindent $r_\sim (\pi(X) ) E - X$ does not belong to the class of $0$ for $\sim$ but belongs to the class of $0$
 for $\widehat{\sim}$.
  
 We then conclude easily that $\sim \leq \widehat{\sim}$, and that
$ \widehat{\sim}$ is different both from $\sim$ and the trivial congruence. This contradicts the maximality of 
$\sim$. Therefore, we have proved that:
$$
\forall X_1 \in \cF/\sim ,\; \exists \lambda \in \R,\; X_1 = \lambda E_1\,.
$$ This $\lambda$ is unique because, $\sim$ being non trivial,  $\cF/\sim$ is a non trivial real vector space.
Now thanks to Proposition \ref{prop:w}.4] we can apply Lemma \ref{lem:ext}.4] to $\cF/\sim$. 
Then we obtain that for any $t,t' \in \R$, $ t E_1\oplus t' E_1= \max( t, t') \, E_1$. Therefore, we can conclude that the map 
$\lambda E_1 \mapsto \lambda$ defines an isomorphism of semi-fields of characteristic $1$ between 
$(\cF/\sim \, , \oplus, +)$ and $(\R, \max ( , ) , +)$. The Theorem is  proved.

\end{proof}


\m 
The next Theorem shows 
that for any $X\in \cF$ one can find a $\phi \in S_E(\cF)$ such that $| \phi(X)|$ takes the maximal value 
permitted by Lemma \ref{lem:char}.1].
\begin{theorem} \label{thm:char} Let $X \in \cF \setminus \{0\}$. Then there exists $\phi \in S_E(\cF)$ such that 
$|\phi(X)|= r(X)$.
\end{theorem}
\begin{proof} At the expense of replacing $X$ by $-X$, we can assume that 
$r(X)$ is the smallest real $\geq 0$ such that $X \leq r(X) E$. Consider the subset  $\cH$
 of $\cF$ defined by:
 $$
 \cH = \{ H \in \cF/\; \exists \lambda \in \R^+,\; -\lambda (r(X) E -X ) \leq H \leq \lambda ( r(X) E -X) \}\,.
 $$ The arguments of the proof of the previous Theorem show that one defines a non trivial congruence 
 $\sim_0$ on $\cF$ by saying that $A \sim_0 B$ if $A-B \in \cH$. By construction $X \sim_0 r(X) E $. 
 By Theorem \ref{thm:max}, there exists a maximal congruence $\sim$ such that 
 $\sim_0 \leq \sim$. Consider the projection $\pi: \cF \rightarrow \cF/\sim$, since $\sim_0 \leq \sim$,  
 one has $\pi(X) =  \pi(r(X) E)$. But $\sim$ is closed, so by Proposition \ref{prop:w}.2]:
 $$
\pi(X)=  \pi(r(X) E) = r(X) \pi (E) \,.
 $$
  Consider next the isomorphism 
 $\Psi: (\cF/\sim, \oplus, +) \rightarrow (\R, \max( ,) , +)$ given by the previous Theorem, by construction 
  $\Psi (\pi(E))=1$. 
 Now, the previous results of this proof show that, by settting 
 $$\phi= \Psi \circ \pi: \cF \rightarrow \R\,,$$ 
 one defines a character of $\cF$ such that $\phi(E)=1$ and $\phi(X) = r(X)$. 
 The Theorem is proved.
\end{proof}
\m
\section{The Classification Theorem \ref{thm:Tilby}. Applications.} $\;$

\m 
In this Section we consider a Banach semi-field $\cF$: it is complete 
for the norm $r(X)$ (cf Definition \ref{def:dist}).
The following classification result is an analogue of the Gelfand-Naimark Theorem which 
classifies the complex commutative $C^*-$algebras. 

\begin{theorem} \label{thm:Tilby} Let $(\cF, \oplus, +)$ be a commutative perfect Banach semi-field of characteristic $1$.
Then the  map 
$$
\Theta: (\cF, \oplus, +)  \rightarrow (C^0(S_E(\cF), \R), {\max}\,( ,), +)
$$
$$
X \mapsto \Theta_X: \;\phi  \mapsto \phi(X)= \Theta_X (\phi)
$$ defines an isometric isomorphism of Banach semi-fields: 
$$\forall X \in \cF, \; r(X)= \sup_{ \phi \in S_E(\cF)} | \Theta_X (\phi) |\,.
$$ Moreover, $\Theta_E\, $ is the constant function ${\bf 1}$: $\phi \mapsto 1$.

\end{theorem}
\begin{proof} 
By the very definition of the topology of $S_E(\cF)$  (see Definition \ref{def:topo}), 
for each $X\in \cF$, $\Theta_X$ is continuous on $S_E(\cF)$. 
Moreover, Lemma \ref{lem:char}.2] shows that  $\Theta$ is $\R-$linear. 

The Definition \ref{def:char}  also shows that for each $\phi \in S_E(\cF)$:
$$
\forall (X, Y)  \in \cF\times \cF ,\; \Theta_{X\oplus Y} (\phi) = \max ( \Theta_X(\phi), \Theta_Y(\phi)), \; {\rm and} \; 
\Theta_E(\phi)=1 \,.
$$
Next, Lemma \ref{lem:char}.1] and Theorem \ref{thm:char} imply that:
$$\forall X \in \cF, \; r(X)= \sup_{ \phi \in S_E(\cF)} | \Theta_X (\phi) |\,.
$$ So we deduce that $\Theta$ defines an injective homomorphim of semi-fields 
which is isometric with respect to the  norms. Since $\cF$ is a Banach vector space for the norm $r(X)$, 
we then conclude that $\Theta( \cF)$ is a closed sub vector space of $C^0(S_E(\cF), \R) $ endowed with 
the supremum norm.
We are going to  apply the Lemma 1 of \cite{Sc}[Page 376] in order to show that $\Theta( \cF)$ is dense 
in $C^0(S_E(\cF), \R) $. Observe that for each $(X, Y, \phi)  \in \cF\times \cF\times S_E(\cF)$:
$$
-\Theta_{-X\oplus -Y} (\phi) =- \max ( \Theta_{-X}(\phi), \Theta_{-Y}(\phi))= \min ( \Theta_X(\phi), \Theta_Y(\phi))\,.
$$ Therefore, if $f,g \in \Theta( \cF)$ then $\min(f,g)$ and $\max (f,g)$ also belong  to $\Theta( \cF)$.
Then we distinguish two cases.

First case: $S_E(\cF)$ is reduced to a point. 

Then $C^0(S_E(\cF), \R) $ is a real line and $\Theta$ is clearly 
an isomorphism.

Second case: $S_E(\cF)$ is not reduced to a point.

 Consider then any two different points $\phi_1$ and $\phi_2$ of $S_E(\cF)$
and also two reals $\alpha, \beta$. Since $\phi_1$ and $\phi_2$ are not equal, we can find 
$Z\in \cF$ such that $a_1= \phi_1(Z) \not=\phi_2(Z) = a_2$. Set:
$$
\lambda = \frac{\alpha - \beta}{a_1 - a_2} ,\; \mu = \frac{-a_2 \alpha + a_1 \beta}{a_1-a_2}\,.
$$ An easy computation shows that:
$$
\Theta_{ \lambda Z + \mu E} (\phi_1)= \alpha,\; \Theta_{ \lambda Z + \mu E} (\phi_2)= \beta\,.
$$ Then,  Lemma 1 of \cite{Sc}[Page 376]  shows that $\Theta( \cF)$ is dense in $C^0(S_E(\cF), \R) $. The Theorem 
is proved.
\end{proof}

Let us briefly examine the case where $\cF$ is not Banach but only  complete for an $F-$norm (cf Definition \ref{def:norm}).
\s 
\begin{proposition}  Assume that  $\cF$  is complete for a $F-$norm.
Then the analogue of the map $\Theta$ of Theorem \ref{thm:Tilby} is   (only) an injective continuous 
homomorphism of semi-fields with dense range.
\end{proposition}
\begin{proof} The proof of Theorem \ref{thm:Tilby} extends here verbatim except at the last stage. Indeed, since 
$\Theta$ is no more isometric when $\cF$ is endowed with a $F-$norm, we can only conclude that the range of $\Theta$ is dense.
\end{proof}

As a first application, we consider the real Banach algebra of Witt vectors $\overline{W}( \cF, E)$ that  Connes-Consani (\cite{CC1}, \cite{C}) constructed and associated 
functorially to $(\cF, E)$. The authors  observed that  the  norm of $\overline{W}( \cF, E)$ is not $C^*$ (\cite{C}) and pointed out that 
its spectrum should contain some interesting information. 

\begin{theorem} \label{thm:Connes} The spectrum of $\overline{W}( \cF, E) \otimes \C$ coincides with $S_E(\cF)$.
\end{theorem}
\begin{proof}
The idea is to first apply Theorem \ref{thm:Tilby} and then \cite{C}[Prop.6.13]. Notice that Connes-Consani wrote the second law multiplicatively 
whereas we wrote it additively. So one needs to use the Logarithm to pass from their view point to ours.
This said, one applies \cite{C}[Prop.6.13] with $\rho(x)= e$ (i.e $T(x) \cong 1$). Using our Theorem \ref{thm:Tilby} one then obtains an
identification of $\overline{W}( \cF, E)$ with $C^0 ( S_E(\cF) , \R)$. The Connes-Consani norm on $\overline{W}( \cF, E)$ 
is not $C^*$ but it is equivalent to the usual $C^*$ supremum norm on $C^0 ( S_E(\cF) , \R)$.
The Theorem  follows then easily.  
\end{proof}

\s

The next result refers to Example \ref{ex:K} of Section 2.1 and determines its spectrum.
\begin{corollary} \label{cor:Tilby} Let $K$ be a compact topological space. Then the spectrum $S_{{\bf 1}}( C^0(K , \R))$ is naturally identified 
with $K$. In other words, for any $\phi \in S_{{\bf 1}}( C^0(K , \R))$ there exists a unique $t\in K$ such that
$\forall X \in C^0(K , \R),\, \phi(X)= X(t)$.
\end{corollary}
\begin{proof} Using Tietze-Urysohn's Theorem (\cite{Sc}[Page 345]), one checks easily that the topology 
of  $K$ coincides with the weakest topology of $K$ making continuous all the elements of 
$C^0(K , \R)$. Of course, any $t\in K$ defines a normalized character of $( C^0(K , \R), \max ( , ), +)$ 
by the formula $\phi_t: X \mapsto X(t)$. Thus, $K \subset S_{{\bf 1}}( C^0(K , \R))$ and the topology of 
$S_{{\bf 1}}( C^0(K , \R))$ induces on $K$ its original topology. So $K$ is a compact subset of $S_{{\bf 1}}( C^0(K , \R))$.
Suppose, by the contrary, that there exists an element $\varphi \in S_{{\bf 1}}( C^0(K , \R))$ which 
does not belong $K$. By Tietze-Urysohn, there exists $F \in C^0\bigl( S_{{\bf 1}}( C^0(K , \R))\,,\, \R\bigr)$ such that 
$\forall t\in K,\; F(\phi_t)= 0$  and $F(\varphi)=1$. By Theorem \ref{thm:Tilby}, there exists $X \in C^0(K, \R)$ such that 
$\Theta_X = F$:
$$\forall \phi \in S_{{\bf 1}}( C^0(K , \R)) , \; \phi(X) =  \Theta_X(\phi)= F(\phi)\,.
$$ Then, for any $t\in K$, we have $X(t)= \phi_t(X)= \Theta_X(\phi_t)=F(\phi_t)=0\,.$
So $X=0$,  which implies $\Theta_X = F=0$. But this contradicts the fact that $F(\varphi)=1$. The Corollary is proved.
\end{proof}

Using the Riesz representation Theorem and integration theory, one can prove 
the previous corollary along the same idea. We could also prove it by combining Theorem \ref{thm:Connes} and \cite{C}.
But we believe  it is interesting to see how it follows, 
in our context, from Theorem \ref{thm:Tilby}.

\s
Here is a simple example where the topology of $S_E(\cF)$ is not metrizable.
\begin{example} Let $\cF_0$ be the Banach semi-field  of  continuous functions $f: [0,1] \rightarrow \R$ such that
$$
\exists C_f >0,\; \forall t \in [0,1], | f(t) | \leq C_f\, t\,.
$$ Consider the  continuous function $E$ on $[0,1]$ defined by $t \mapsto t= E(t)$.
Then the previous Theorem applies to $\cF_0$ with $S_E(\cF_0)$ equal to the Stone-Cech compactification of 
$]0,1]$.
\end{example}
\s

\s
We now  state several consequences, some of them being motivated by some Banach nonarchimedean theory \cite{KL}[Section 2.3].

\begin{proposition} \label{prop:=} Let $\sim$ be a  non trivial congruence of $\cF$. Then there exists 
$\phi \in S_E ( \cF)$ such that:
$$
\forall X, Y \in \cF,\; ( X \sim Y )\, \Rightarrow \, \phi(X)=\phi(Y)\,.
$$
\end{proposition}
\begin{proof} This is an easy consequence from Theorems \ref{thm:max} and \ref{thm:gm}. 
\end{proof}

\s
\begin{definition} Let $X$ be an  element of $ \cF$.
\item 1] $X$  is said to be regular if  there does not exist any non trivial congruence $\sim$ for which 
$0 \sim X$.
\item 2] $X$ is said to be absorbing if for any $Z \in \cF$ there exists $n\in \N$ such that: $ Z \leq n X$.

\end{definition}
The next Proposition gives a characterization of some algebraic notions internal to $\cF$
in terms of the normalized characters. Its proof uses Theorem \ref{thm:Tilby} and is left to the reader.
\begin{proposition} \label{prop:cons} Let $X \in \cF$. Then the following are true.
\item 1] $0 \leq X$ if and only if $\forall \phi \in S_E ( \cF)$, $0\leq \phi(X)$.
\item 2] $X$ is regular if and only if $\forall \phi \in S_E ( \cF)$, $\phi(X)\not=0$.
\item 3] $X$ is absorbing if and only if $\exists \epsilon >0$, $\forall \phi \in S_E ( \cF)$, $ \epsilon < \phi(X)$.
\end{proposition}

\s Now we apply Corollary \ref{cor:Tilby} to the determination of the spectrum of an  interesting geometric example, which is in fact at the origin of our interest for this topic.
\begin{theorem} \label{thm:conv} Let $F$ be a finite dimensional real vector space and denote by $\cC$ the set of all the compact convex subsets of $F$ which contains $0$. 
Denote by ${\rm conv} ( A \cup B)$ the convex hull of $A \cup B$ for  $A, B \in \cC$.
Consider a map $\phi: \cC \rightarrow \R$ such that:
$$
\forall A, B \in \cC, \;\phi (A+B)= \phi(A) + \phi(B),\; \phi ({\rm conv}\, ( A \cup B) )= \max ( \phi(A), \phi(B))\,,
$$ where  $+$ denotes the Minkowski sum. Then there exists  a linear form $\psi \in F^*$ such that
$$
\forall A \in \cC,\; \phi(A) = \max_{t \in A} \psi(t)\, = l_A(\psi)\,.
$$ Recall that $l_A(\psi)$ is the value at $\psi$ of the so called support function of $A$.
 Since $0\in A$, $l_A (\psi)$ is automatically $\geq 0$.
\end{theorem}
\begin{proof} We can assume that $\phi$ is not identically zero. Since every element $B$ of $\cC$ contains $0$, one observes that $\phi ( \{0\})=0$ 
and that: 
$$
\phi ({\rm conv}\, ( \{0\} \cup B) )= \phi(B)= \max ( \phi( \{0\}) , \phi(B))\,.
$$ Therefore,  $\phi (B) \geq 0$.
Fix an euclidean norm $N$ on $F$ with scalar product $< , >$ and closed unit 
ball $E$. 
At the expense of replacing $\phi$
by $\lambda \phi$ for a suitable $\lambda >0$ we can assume that $\phi(E)=1$. 
Next, recall that $l_E(\psi)$ is  the operator norm (associated to $N$)
of the linear form $\psi \in F^*$.  Consider the   sphere of $F^*$:
$$
S_E = \{ \psi \in F^*/\; l_E(\psi)= 1\}\,.
$$

Observe that $(\cC, \oplus= {\rm conv}\, (   \cup  ), +)$ is a semi-ring and that for each $A\in \cC$ 
we can find a smallest real $t \geq 0$ such that $A \leq t E$ (i.e $A \subset t E$). Call it $r(A)$, then 
$r(A)=0$ iff $A=\{0\}$. Explicitly,  $r(A)$ is the greatest euclidean norm $ N(v)$ for $v \in A$.

\begin{lemma} For any $(\lambda, A) \in \R^+ \times \cC$ one has: $\phi( \lambda A) = \lambda \phi(A)$ and
$\phi (A) \leq r(A)$.

\end{lemma}
\begin{proof} Set $\psi( \lambda) = \phi( \lambda A)$ for $\lambda \in \R^+$. One observes that $\psi$ defines an additive non decreasing map
from  $\R^+$ to $\R^+$. Therefore, for any $\lambda \geq 0$, $\psi(\lambda)= \lambda \psi(1)$. This proves the first equality. 

Next consider $A \in \cC$, clearly $A \oplus r(A) E= r(A) E$. By applying $\phi$ one obtains:
$$\phi(A \oplus r(A) E) = \max( \phi(A), \phi(r(A) E)= \phi(r(A) E) \,.$$
 But by the first part, $\phi(r(A) E) = r(A) \phi(E)= r(A)$.
This proves the Lemma.

\end{proof}
Next denote by $\cR$  the set of convex functions $g$ on $F^*$ with values 
in $[0, + \infty[$ such that $g(\lambda \psi) = \lambda g(\psi)$ for any $(\lambda,\psi)\in \R^+ \times F^*$.
Endow $\cR$ with the following norm:
$$
\forall g \in \cR,\;\; \| g \|= \sup_{\psi \in S_E} | g(\psi) |\,.
$$

\begin{lemma} \item 1]
The support function $A \mapsto l_A$ 
defines an isomorphism of semi-rings from $(\cC, \oplus, +)$ onto 
$(\cR, \max( , ), +)$.  

\item 2] $\forall A \in \cC,\; r(A)\,=\, \|  l_A\| \,.$
 
\end{lemma}
\begin{proof} 1]  is a well known consequence of  Hahn-Banach.

2] The euclidean scalar product $< , >$ of $F$, fixed above, allows to identify $S_E$ with the euclidean sphere $S$ of 
$F$. Using these identifications, one finds that for any $A \in \cC$:
$$\|  l_A\| = \sup_{u \in S}\, | \max_{a \in A} < u , a>  | \,=\,\sup_{u \in S, \, a \in A} <u , a > \,.
$$ By Cauchy-Schwartz (and $S=-S$), $\|  l_A\|$ appears now to be the greatest euclidean norm $N(a)$ of the elements of $A$.
But this is precisely $r(A)$. The Lemma is proved.
\end{proof}

Thus the character $\phi$  of $\cC$ induces a character, still denoted $\phi$, on $\cR$; with this identification 
one has $ \phi(A) = \phi(l_A)$ for $A \in \cC$.  Notice, by definition, 
that the restriction of $l_E$ to the sphere $S_E$ is the constant function ${\bf 1}$ and that $\phi(E)=\phi(l_E)=1$.
The two previous Lemmas then imply:
$$
\forall g \in \cR,\;\;  | \phi(g) | \leq \| g \|\,.
$$ In the rest of this proof we identify,  via the map $ g \mapsto g_{| S_E} $, $\cR$ with a sub semi-ring 
of $C^0(S_E ; \R)$.
Since $\cR$ is clearly cancellative, we can consider its  semi-field of fractions  $(\cG_0, \max ( , ), +) $. It has  the structure of a real vector space and is given by: 
$$\cG_0= \{g_0 -g_1\,/\,  g_0,g_1 \in \cR\} \subset C^0( S_E , \R) \,.$$

\begin{lemma} By setting $\phi (g_0 -g_1)= \phi(g_0) -\phi(g_1)$ for any $g_0,g_1 \in \cR$,  one defines in 
an intrinsic way a character 
  still denoted $\phi$, of  $(\cG_0, \max( , ), +) $, which is $\R-$linear. It satisfies:
\begin{equation} \label{eq:conv}
\forall g \in \cG_0,\;\;  | \phi(g) | \leq \| g \| \, = \max_{\psi \in S_E} | g(\psi) |\,.
\end{equation}
\end{lemma}
\begin{proof} The (intrinsic) extension of $\phi$ as a $\R-$linear character to $\cG_0$ is left to the reader. One has:
$$
\forall g \in \cG_0 , \; - \| g \| {\bf 1} \leq g \leq \| g \| {\bf 1}\,.
$$ Since under the previous identifications $\phi ( {\bf 1}) =1$, one obtains \eqref{eq:conv} by applying 
the character $\phi$ to the previous inequality in $\cG_0$.

\end{proof} 
 Observe that the semi-field $\cG_0$ is not complete for the norm $\| g \|= \max_{\psi \in S_E} | g(\psi) |$. Next, consider 
 two different points $\psi_1, \psi_2$ of $S_E$ and denote by $u_1$ the (unitary) vector of $F$ such that 
 $< u_1 , z >\, =\, \psi_1(z) $ for all $z \in F$.
 Set  $A_1 = [0 , u_1] \subset F$, it is  clear that 
 $l_{A_1} (\psi_1) \not=l_{A_1} (\psi_2)$.
Then the arguments of the end of the proof of Theorem \ref{thm:Tilby} allow to show that $\cG_0$ is dense in 
$C^0( S_E, \R) $. Then inequality \eqref{eq:conv} implies that $\phi$ can be extended to a
 character, still denoted $\phi$,  of $(C^0( S_E, \R) , \max ( , ),  +)$.
 By Corollary \ref{cor:Tilby} there exists $\psi \in S_E$ such that for any $g \in C^0( S_E, \R)$,
 $\phi (g)= g(\psi)$. Going back to $\cC$ we obtain, using the previous identifications:
 $$
 \forall A \in \cC,\; \phi (A)= \phi(l_A)= l_A(\psi)= \max_{v \in A} \psi(v)\,.
 $$  Theorem \ref{thm:conv} is proved.
\end{proof}

\m
\section{Determination of the closed congruences of $(\cF, \oplus, +)$.} $\;$

\s In this Section we fix a Banach semi-field $\cF$ (see Definition \ref{def:dist}) and  a closed congruence $\sim$ on $\cF$.  Consider 
then the natural projection 
 $\pi: \cF \rightarrow \cF/\sim$  and set $E_1 = \pi(E)$. Our goal is to provide a geometric description 
of $\sim$, but since we have only an inequality in Proposition \ref{prop:w}.3], we have to be careful.

\s
 Recall that  $S_{E_1}( \cF/\sim)$ denotes the set of characters of $ (\cF/\sim, \oplus , +) $
satisfying $\phi(E_1)=1$. By Theorem \ref{thm:comp},   $S_{E_1}( \cF/\sim)$ is compact for the topology $\cT$, the weakest one making continuous 
all the maps $\phi \mapsto \phi(X_1)$ where $X_1$ runs over $\cF/ \sim$.
The   next Proposition 
identifies $S_{E_1}( \cF/\sim)$ with a compact subset of $(S_{E}( \cF) , \cT)$.
\begin{proposition} \label{prop:j}
 The map $\phi \mapsto \phi \circ \pi$ defines an injective map 
$j: S_{E_1}( \cF/\sim) \rightarrow S_{E}( \cF)$. Endow $j(S_{E_1}( \cF/\sim) )$ with the induced topology of $(S_{E}( \cF) , \cT)$.
Then $j$ induces  a homeomorphism from 
$ S_{E_1}( \cF/\sim)$ onto $j(S_{E_1}( \cF/\sim) )$.  Thus we shall view freely $ S_{E_1}( \cF/\sim)$ 
as a compact subspace of  $S_{E}( \cF)$.

\end{proposition}
\begin{proof} 
This is an easy consequence of the definition of the topology $\cT$ (Definition \ref{def:topo})  and of the fact that 
$\pi: \cF \rightarrow \cF/\sim$ is a surjective homomorphism of semi-fields.
\end{proof}

\s
The  following Theorem  classifies all the closed congruences on $\cF$ and is the main result of this Section.
\begin{theorem} \label{thm:Tilby'} Set $K_1= S_{E_1}( \cF/\sim)$. The following are true.
\item 1]
$$\forall X,Y \in \cF,\; X \sim Y \Leftrightarrow (\Theta_X)_{|_{j(K_1)}}= (\Theta_Y)_{|_{j(K_1)}}\,,
$$ where the map $\Theta$ is defined in Theorem \ref{thm:Tilby}. 

\item 2] The map
$$
\Theta^1: (\cF/\sim\,, \,\oplus, +) \rightarrow (C^0(K_1 , \R) ,\, \max ( , ) , \,+)
$$
$$
X_1 \mapsto \Theta^1_{X_1}: \; \phi \mapsto \phi(X_1)= \Theta^1_{X_1} (\phi)
$$
defines an isometric isomorphism of Banach semi-fields:
$$
\forall X_1 \in \cF/\sim\,, \; \;\inf_{X \in \cF,\, \pi (X)= X_1}\, r(X)\,= \max_{\phi \in K_1} | \Theta^1_{X_1}(\phi) |\,
 = \,r_\sim (X_1) \,. 
$$
\item 3] Lastly, 
$$
\forall \phi \in K_1 , \; \Theta_{\pi(X)}^1 ( \phi) = \Theta_X (\phi \circ \pi)\,.
$$ In other words, under the identification of Proposition \ref{prop:j}, $\Theta_{\pi(X)}^1$ is the restriction of $\Theta_X$ to $j(K_1)$.
\end{theorem}

\begin{proof} We begin with preliminary observations.
By  Proposition \ref{prop:w}, applied with $\cF$  endowed with the complete $F-$norm $r(X)$, $\cF/\sim$ is complete for the $F-$norm $\| \, \|_1$ and $r_\sim$ is a norm.
Define $\Theta^1_{X_1}$ as in Part 2] of the Theorem.

One obtains the  inequality: 
\begin{equation} \label{eq:prel}
\forall X_1 \in \cF/\sim\;, \;  \max_{\phi \in K_1} | \Theta^1_{X_1}(\phi) |\,
 = \,r_\sim (X_1) \; \leq \inf_{Y \in \cF,\, \pi (Y)= X_1}\, r(Y) \, = \|  X_1 \|_1, 
 \end{equation} by the following arguments.
 First, apply the inequality of Proposition \ref{prop:w}.3]. Second, apply Lemma \ref{lem:char}.1] and Theorem \ref{thm:char} 
 to $\cF/\sim$ instead of $\cF$, and then \eqref{eq:prel} follows immediately.
 
Let us now prove 1]. It suffices to prove the result for $Y=0$.
Let $X \in \cF$ then Theorem \ref{thm:char}, applied to $\cF / \sim$ and $\pi(X)$, shows that 
$\pi(X)=0$ (or $X \sim 0$)  if and only if $\phi\circ \pi(X)=0$ for all $\phi \in K_1$. 
But $\phi\circ \pi(X)=0$ means precisely that $\Theta_X (\phi \circ \pi) = \Theta_X( j(\phi))=0$ where $\Theta_X$ is defined in Theorem \ref{thm:Tilby}. By Proposition \ref{prop:j} 
which explains the way $K_1$ is embedded in $S_E( \cF)$, 
we then conclude that  
$\pi(X)=0$ if and only if $\Theta_X$ vanishes on $j(K_1)$. 
The result is proved.

Let us prove 2]. By the inequality \eqref{eq:prel} and the arguments of the proof of Theorem \ref{thm:Tilby}, the
the map:
$$
\Theta^1: (\cF/\sim\,, \,\oplus, +) \rightarrow (C^0(K_1 , \R) ,\, \max ( , ) , \,+)
$$
$$
X_1 \mapsto \Theta^1_{X_1}: \; \phi \mapsto \phi(X_1)= \Theta^1_{X_1} (\phi)
$$
defines a continuous injective  homomorphism of  semi-fields
with dense range in $C^0(K_1 , \R)$. 
Let us now show that $\Theta^1$ is both surjective and isometric.

Consider $F \in C^0(K_1 , \R)$; the homeomorphism 
$j:K_1 \rightarrow j(K_1)$ of Proposition \ref{prop:j} allows to consider $F\circ j^{-1}  \in C^0(j(K_1) , \R)$. By Urysohn (\cite{Sc}[page 347]) applied to $j(K_1) \subset S_E(\cF)$, there exists $G \in C^0( S_E(\cF), \R)$ such that:
\begin{equation} \label{eq:u}
G_{| j(K_1)}= F \circ j^{-1},\; \;  {\rm and} \; \; \max_{v \in j(K_1)} | F\circ j^{-1}(v) | \,=\, \max_{w \in S_E(\cF)} | G(w) | \,.
\end{equation} By Theorem \ref{thm:Tilby},  $G= \Theta_X$ for a suitable $X \in \cF$, and
 $r(X)= \max_{ w\in S_E(\cF)} | G(w) |$.   Thus we obtain:  $(\Theta_X)_{| j(K_1)}= F\circ j^{-1}$. By Proposition \ref{prop:j}, 
  this means that for any character $\phi \in K_1$ one has, with $j(\phi)= \phi \circ \pi$:
 $$\phi \circ \pi(X)\,=\, \Theta_X( \phi \circ \pi)\, = \,F\circ j^{-1} ( \phi \circ \pi)= F(\phi) \,.$$
 In other words, we obtain $F= \Theta_{\pi(X)}^1$. Moreover, Lemma \ref{lem:char}.1] and Theorem \ref{thm:char}, applied 
 to $\pi(X) \in \cF/\sim$, show that 
 $r_\sim (\pi(X))= \displaystyle \max_{ \phi \in K_1} | F(\phi) |$. Using \eqref{eq:u}   we then conclude that:
  $$
  r_\sim (\pi(X)) = \max_{ \phi \in K_1} | F(\phi) | \, = \, \max_{w \in S_E(\cF)} | G(w) | \,=\,
  r(X) \,.
  $$ But, by  
   \eqref{eq:prel} one has:

  $$
 r_\sim (\pi(X)) \, \leq \, \inf_{Y\in \cF, \, \pi(Y)= \pi(X)} r(Y) \, \leq r(X)\,.
 $$ Therefore, the two previous inequalities are in fact equalities.
 Hence, $\Theta^1$ is an isometry and moreover we have just seen that it is surjective.
 Lastly, 3] now becomes obvious. The Theorem \ref{thm:Tilby'} is proved.
\end{proof}

Therefore, the inequality of Proposition \ref{prop:w}.3] is actually an equality !.

\m
\section{About commutative perfect cancellative semirings of characteristic $1$.} $\;$

\s

\subsection{Cancellative semi-rings and congruences.} $\;$

\m
In this Section, $\cR$ will denote a cancellative semi-ring as in the following general definition.
We shall apply the tools constructed in the previous sections to the study of $(\cR, \oplus, +)$.
\begin{definition} \label{def:gen}
A  commutative  semi-ring of characteristic $1$ 
is  a triple $(\cR, \oplus, +)$ where $\cR$ is a set endowed with two associative laws $\oplus, +$  such that 
 the following two conditions are satisfied: 
 \item 1] $(\cR, +)$ is an abelian monoid (with identity element $0$).
 
\item 2] $(\cR, \oplus)$ is a commutative  monoid such that 
$$\forall X \in \cR, \; X \oplus X= X, \; \; {\rm and},
$$ 
$$
\forall X,Y, Z \in \cR,\; X+ (Y\oplus Z)= (X+Y) \oplus (X+Z)\,.
$$
\item 3] $(\cR, \oplus, +)$ is said to be cancellative if 
$$
\forall X,Y, Z\in \cR,\; \,X+Y=X+Z \Rightarrow Y=Z\,.
$$
\item 4] $(\cR, \oplus, +)$ is said to be perfect if for any $n \in \N^*$, the map $\theta_n:\, X \mapsto n X$ is surjective from $\cR$ onto $\cR$.
\end{definition}
\s
We recall the partial order $\leq $ on $\cR$ defined by
$X\leq Y$ if $X\oplus Y=Y$.

\s
Let us first give two concrete examples.

\begin{example} Denote by $\cR_0$ the set of piecewise affine convex functions from $[0,1]$ to $ \R$ 
sending $[0,1] \cap \Q$ into $\Q$. They are of the form:
$$
v \mapsto \max ( a_1  v +b_1, \ldots , a_n v +b_n) ,\; {\rm where\; all\;  the} \; a_j\,,  b_j \in \Q\,.
$$
 Then $(\cR_0, \max ( ,), +)$ is a commutative cancellative perfect 
semi-ring.  
\end{example}
\s 
\begin{example} \label{ex:nf} Let $K$ a number field, 
consider the canonical embedding:
$$ \sigma (x) = (\sigma_1 (x), \ldots,\sigma_{r_1}(x), \dots , \sigma_{r_1 + r_2}(x) ) \in  \R^{r_1} \times \C^{ r_2} \,,
$$ where $(\sigma_1 , \ldots , \sigma_{r_1})$ and $(\sigma_{1+r_1}, \ldots, 
\sigma_{r_1 + r_2}, \overline{\sigma_{1+r_1} }, \ldots , \overline{\sigma_{r_1+r_2} })$
denote respectively the list of  real and complex field embeddings of $K$ 
 (see \cite{S}[page 68]). Denote by $\cC_K$ the set of compact convex polygons 
 of $\R^{r_1} \times \C^{ r_2}$ whose extreme points belong to $\sigma (K)$.

\noindent Then 
 $(\cC_K, \oplus= {\rm  conv}\, (  \cup ), +)$ is a   commutative cancellative perfect 
semi-ring.  Its semi-field of fractions is given by
$$
\cH_K = \{ l_A - l_B / \, A, B \in \cC_K\} \, 
$$ where $l_A$ denotes the support function of $A$ (cf Theorem \ref{thm:conv}).

\end{example}
\s
We now recall the definition of a congruence for a semi-ring $\cR$. It is the analogue of the notion of 
ideal in Ring theory. 
\begin{definition} \label{def:cong} Let $\cR$ be a commutative cancellative semi-ring of characteristic $1$.
\item 1] A congruence $\sim$ on $\cR$ is an equivalence relation on $\cR$
satisfying the following condition valid for all $X,Y, Z \in \cR$: 

  \item If $X \sim Y$ then   $X+Z \sim Y+Z$, and $X\oplus Z \sim Y\oplus Z$.
\item 2] The  congruence 
$\sim$  is said to be cancellative if for any $X,Y, Z \in \cR$, 
$$X + Y \sim X +Z\, \Rightarrow Y\sim Z\,.
$$
 Observe that  for a cancellative congruence $\sim$,   the class of $0$  (denoted $[0]$), 
 is a cancellative sub semi-ring of $\cR$. 

\end{definition}
It is well known (\cite{Golan}) that if  $\sim$ is a cancellative congruence (as above)  then the quotient semi-ring $\cR/\sim$ is  cancellative.

\s Every cancellative semi-ring $\cR$ is canonically embedded in its semi-field  of fractions $\cF$ so that $\cF =\{ X-Y/\; X, Y \in \cR\}$ (\cite{Golan}). 
Therefore, it is interesting to examine to which extent a cancellative congruence on $\cR$ can, in some sense, be extended to $\cF$.

\begin{proposition} \label{prop:corr} \item 1] Let $\sim$ be a cancellative congruence on $\cR$. Then, one defines in an intrinsic way 
a congruence $\sim_\cF$ on $\cF$ (see Definition \ref{def:congF}) by saying that:
$$
\forall A,B,A',B' \in \cR,\; \; A-B \, \sim_\cF \,A'-B' , \;\, {\rm iff}\; A+B' \,\sim\, A' + B\,.
$$ The class of $0$ for $\sim_\cF$  is given by $\pi_{\sim_\cF}(0) = 
\{A -B / \, A, B \in \cR, \, A \sim B \}$.

Moreover, for any $A, B \in \cR$, $A \sim B \Leftrightarrow A \sim_\cF B$. 

\item 2] Conversely, let $\sim_1$ be a congruence on $\cF$. Then there exists a unique cancellative congruence 
$\sim$ on $\cR$ such that $\sim_\cF= \sim_1$. Basically, $\sim$ is the restriction of $\sim_1$ to $\cR$.

\item 3] The natural map $\cR/ \sim \, \, \rightarrow \cF / \sim_\cF$ is an injective homomorphism of semi-rings and identifies $\cF / \sim_\cF$ with  the semi-field of fractions 
of $\cR/ \sim$.
\end{proposition}
\begin{proof} 1] First let us show that $\sim_\cF$ is well defined. Suppose that $A_1-B_1= A-B$ for two other elements $A_1, B_1$ of 
$\cR$. Assume moreover that  $A+B' \,\sim\, A' + B$ and let us show that $A_1+B' \,\sim\, A' + B_1$. 
By the congruence properties of  $\sim$,  we  deduce:
$$
A+B' + B_1 \,\sim\, A' + B + B_1 \,.
$$ Since $A_1 + B = A + B_1$, we can replace $A + B_1$ by $A_1 + B$ in the left hand side. One obtains: $A_1 + B + B' \, \sim A' + B + B_1$. 
By the cancellativity of  $\sim$, we can simplify by $B$ and get: $A_1 +  B' \, \sim A' + B_1 $ as desired. Going a bit further along this argument, one obtains that $\sim_\cF$  is intrinsically defined that way and that it is an equivalence relation. Moreover, 
it becomes clear that for any $A, B \in \cR$, $A \sim B \Leftrightarrow A \sim_\cF B$. We shall use this fact  soon.

Now let  $ A,B,A',B' \in \cR$ be such that $A+B' \,\sim\, A' + B$.  Consider then $U, V \in \cR$. Let us show that:
\begin{equation} \label{eq:proplus}
(A-B) \oplus (U- V) \, \sim_\cF \, (A'-B') \oplus ( U-V) \,.
\end{equation} By adding $B+B' + V$ to both sides and using the very definition of $\sim_\cF$, one obtains that \eqref{eq:proplus} is equivalent to:
$$
(A+ B'+ V) \oplus (U+ B+B') \, \sim \, (A'+B +V) \oplus ( U+ B+B') \,.
$$ But this is true by the hypothesis  $A+B' \,\sim\, A' + B$ and the fact that $\sim$ is a congruence. Lastly, it is clear 
that $A-B \sim_\cF A'-B'$ implies:
 $$B-A \sim_\cF B'-A' , \; {\rm and}\; \; \forall U, V \in \cR , \,\; A-B + U-V  \sim_\cF A'-B' +U-V\,.
 $$
We have thus proved that 
$\sim_\cF$ is a congruence.

2] and 3] are easy and left to the reader.

\end{proof}


\s
Notice that in  general $\pi_{\sim_\cF}(0)$ does not coincide with the semi-field of fractions $[0]_\cF$ of the class of $0$ 
for $\sim$.
A counter example 
is provided by the semi-ring $(\cR_0, \max( , ) , +)$ of all the convex functions on $[0,1]$ with values in $\R^+$. The congruence 
$\sim$ being defined by $f \sim g$ if $f=g$ on $[1/5 , 2/5] \cup [3/5 , 4/5]$. The point is that  any element 
of $[0]_\cF$ vanishes identically on $[1/5 , 4/5]$.

The next Proposition, whose proof is left to the reader,  gives a sufficient condition for this to be true

 \begin{proposition} \label{prop:cong} We keep the notations  of  the previous Proposition and  assume moreover that $\sim$ satisfies the following additional hypothesis:
$$\forall X,Y \in \cR,\; \Bigl( X\sim Y \; {\rm and}\, X\leq Y \Bigr)\, \Rightarrow \, \exists Z \in \cR, \; X+Z=Y\,.
$$
Then $\pi_{\sim_\cF}(0)$ coincides with the semi-field of fractions of the class of $0$ for $\sim$.




%
\end{proposition}

\s
\subsection{ Embeddings of a large class of abstract cancellative semi-rings into semi-rings of continuous functions.} $\;$

\m

Now, we utilize  the perfectness property  for a semi-ring. 

 {\it Until the end of this Section, we shall assume that $(\cR, \oplus, +)$ is a perfect commutative cancellative semi-ring of 
characteristic $1$.} Observe that its semi-field of fractions $\cF$ is  perfect too.

Then 
we know from \cite{Golan}[Prop.4.43] and \cite{C}[Lemma 4.3] that the maps $\theta_n$ of the  Definition \ref{def:gen}.4] are bijective for all 
$n\in \N^*$ and also that:

\begin{lemma} \label{lem:alain} \cite{C}[Prop. 4.5]
 The equality 
$$\theta_{a/b}= \theta_a \circ \theta_b^{-1}, \; (a,b) \in \N^* \times \N^*\,,$$
defines an action of $\Q^{+*}$ on $(\cR, \oplus, +)$ satisfying:
$$ \forall t, t' \in \Q^{+*} , \;
\theta_{t t'}= \theta_t \circ \theta_{t'} ,\; \;\forall X \in \cR,\; 
\theta_t(X) + \theta_{t'}(X) = \theta_{t + t'}(X)\,.
$$ In the sequel, we shall write $t X$ instead of $\theta_t(X)$.
\end{lemma}

\s 

We shall also assume until the end of  this Section  that $\cR$ satisfies the following (cf \cite{C}):
\begin{assumption} \label{ass:c}  There exists $E \in \cR$ such that:  
$$\forall X \in \cR, \; \, \exists t \in \Q^+,\,\; - tE\leq X \leq t E\,.$$
 
\end{assumption}

Of course, these inequalities hold in the semi-field $\cF$ of fractions of $\cR$ and they imply clearly that for any $Z \in \cF$ there exists $t \in \Q^+$ such that
$- t E \leq Z \leq t E$.
\s

We leave to the reader the proof of the next Lemma, it follows easily from standard arguments used in Section 3.
\begin{lemma} One defines a cancellative congruence $\sim$ on $\cR$ by saying that
$X \sim Y $ if:
 $$\forall t \in \Q^+,\;\; -t E \leq X-Y \leq t E\, ,$$
 where of course these inequalities hold in  $\cF$. 
\end{lemma}

\s 
At the expense of replacing $\cR$ by  $\cR/\sim$ we may, {\bf and shall}, assume also the following assumption for the rest of this Section:
\begin{assumption} \label{ass:d}
For any $X, Y \in \cR$, set:
 $$r(X-Y) = \inf\, \{ t \in \Q^+/\; 
-t E \leq X-Y \leq t E \} \,.
$$  Then,  $r(X-Y)= 0$ implies that $X=Y$. 

\end{assumption}
 It is clear  
 that Assumptions \ref{ass:c} and \ref{ass:d} imply that $\cF$ satisfies Assumptions 1 and 2. Therefore,   
by Lemma \ref{lem:dist}, $(X,Y) \mapsto r(X-Y)$ defines  a distance on $\cF$ and, by restriction, on $\cR$.
Next we consider the set of normalized characters on $\cR$.

\s
\begin{lemma} Denote by $S_E(\cR)$ the set of semi-ring homomorphisms (called characters) $\phi: (\cR, \oplus, +) \rightarrow 
(\R,  \max ( , ), +)$ satisfying $\phi (E)=1$. Endow $S_E(\cR)$ with the weakest topology (called $\cT$) rendering continuous all
the maps $\phi \mapsto \phi(X),\, X \in \cR$. 

Then the restriction map
$ 
\Xi: S_E(\cF) \rightarrow S_E(\cR)
$ induces a homeomorphism between compacts sets for the topologies $\cT$.
\end{lemma} 
\begin{proof} Observe that $(S_E(\cR) , \cT)$ is Hausdorff, and recall that $S_E(\cF)$ is compact by Theorem \ref{thm:comp}. Let $\phi \in S_E(\cR) $, let us show that 
it can be uniquely extended to an element of  $S_E(\cF)$. Consider an element $X \in \cF$, write it as 
$$
X= A-B = A'-B' , \; {\rm for}\; A,B,A',B' \in \cR\,.
$$ By applying $\phi$ to the equality $A+B' = A'+B$, one checks that $\phi(A) -\phi(B)= \phi(A')-\phi(B')$; we call this number $\phi(X)$. This defines an additive map $\phi: (S_E (\cF) , +) \rightarrow (\R ,+)$

Next writing:
$$
(A_1-B_1) \oplus (A_2-B_2)= (A_1 +B_2) \oplus (A_2 + B_1) - B_1- B_2\, , 
$$ where the $A_j, B_j$ belong to $\cR$,  one checks that:
$$
\phi \bigl( (A_1-B_1) \oplus (A_2-B_2) \bigr) = \max ( \phi (A_1-B_1) , \phi (A_2-B_2) )\,.
$$ This shows that $\phi$ can be (uniquely) extended to an element of $S_E(\cF)$. 
Therefore, $\Xi$ is a bijection. It is clearly continuous, then by the compactness of $S_E(\cF)$, 
$\Xi$ is a homeomorphism. One can also check this last fact by hands.
\end{proof} 

\s 
\begin{example} With the notations of Theorem \ref{thm:conv}, 
$S_E( \cC)$ is an euclidean sphere of $F^*$.
\end{example}

\s

The next Theorem provides a natural semi-ring embedding of $\cR$ (satisfying Assumption \ref{ass:d}) as 
a sub space of continuous functions on the compact space $S_E(\cR)$.
\begin{theorem} \label{thm:mapj} One defines an injective homomorphism $j$ of semi-rings in the following way:
$$
j: (\cR, \oplus, +) \rightarrow (C^0(S_E(\cR), \R) \max ( ,), +)\,,
$$
$$
X \mapsto j_X\,:\, \phi \mapsto \phi(X)= j_X ( \phi)\,.
$$
\end{theorem}
\begin{proof}  We just give a short sketch,  the interested reader can fill in the details by using the material of Section 2.1.
One checks that the completion $\widehat{\cF}$ of $\cF$ for the metric $r(X-Y)$ is endowed naturally with the 
structure of a commutative Banach perfect semi-field of characteristic $1$. Moreover, the injection $i: \cF \rightarrow 
\widehat{\cF}$ defines a homomorphism of semi-fields. 
\begin{lemma} \label{lem:extchar} The restriction map $\Xi: S_E(\widehat{\cF} ) \rightarrow S_E({\cF} )$ 
defines a homeomorphim between compacts sets endowed with the topology $\cT$.
\end{lemma}
\begin{proof} Consider $\phi \in S_E({\cF} )$  and let us show briefly  that it can extended uniquely 
to an element of $S_E(\widehat{\cF} )$. Let $X \in \widehat{\cF}$ and $(X_n)$ a sequence of points of 
$\cF$ converging to $X$. Then there exists a sequence of positive rationals $(t_n)$ such that $\lim_{ n \rightarrow + \infty} 
t_n = 0$ and:
$$
\forall n, p \in \N,\; -t_n E \leq X_n - X_{n+p} \leq t_n E \,.
$$ By the properties of $\phi \in S_E({\cF} )$, one gets $ -t_n \leq \phi (X_n) - \phi(X_{n+p}) \leq t_n$. 
Thus for any sequence $(X_n)$ in $\cF$ converging to $X$, the sequence of reals  $(\phi (X_n))$ converge. 
It is standard that this limit does not depend on the choice of $(X_n)$, call it $\phi(X)$. Using the material of Section 2.1, one checks right away 
that this extension $\phi: \widehat{\cF} \rightarrow \R$ is a normalized character. Therefore $\Xi$ is a bijection, which is obviously continuous.
By the compactness of $S_E(\widehat{\cF} )$, $\Xi$ is a homeomorphism. This last fact can also be proved by hands.

\end{proof}
By the two previous Lemmas,  we have natural homeomorphisms: 
$$
 \ S_E(\widehat{\cF} ) \simeq S_E({\cF} ) \simeq S_E({\cR} )\,.
$$ 
Then,  
 the Theorem follows  
  by applying 
Theorem \ref{thm:Tilby} to $\widehat{\cF}$ , where   $j$  denotes the restriction of $\Theta$ 
to $\cR$ $( \subset \widehat{\cF})$.
 \end{proof}

One obtains the following  analogue of  \cite{KL}[Cor.2.3.5] from non archimedean theory.
\begin{corollary} \label{cor:ked} Let $\sim$ be a non trivial cancellative congruence on $\cR$.
 Then, there exists $\phi \in S_E( \cR)$, such that:
$$
\forall X,Y \in \cR,\; X \sim Y\, \Rightarrow \phi(X) = \phi(Y)\,.
$$
\end{corollary}
\begin{proof} By Proposition \ref{prop:corr} the congruence $\sim_\cF$ is not trivial, denote by $\cH$ the class of $0$ (for $\sim_\cF$) in $\cF$. Introduce as above the Banach completion $\widehat{\cF}$ of $\cF$ for the metric $r(X-Y)$.
Consider then the  closure $\overline{\cH}$ of $\cH$ in $\widehat{\cF}$. 
Using  Lemma \ref{lem:cong}, one checks easily that $\overline{\cH}$ is a sub semi-field of $\widehat{\cF}$.
Consider $(X,Y,H) \in \widehat{\cF} \times \widehat{\cF} \times \overline{\cH}$. We are going to show the existence 
of $H' \in \overline{\cH}$ such that:
\begin{equation} \label{eq:closcong}
(X+H) \oplus Y = X \oplus Y \, \, +H'\,.
\end{equation} By Lemma \ref{lem:cong}.2], this will show that the relation  defined by $U \,\overline{\sim_\cF} \, V$ if 
 $U-V \in \overline{\cH}$ is a congruence. Now, by Lemma \ref{lem:cong}.2] applied to $\cF$ 
  there exists a sequence $(X_n, Y_n, H_n)_{n \in \N}$  of $\cF\times \cF \times \cH$ converging  to 
 $(X,Y,H)$ in $\widehat{\cF}^3$ such  the following is true:
  $$
 \forall n \in \N,\; (X_n +H_n) \oplus Y_n= X_n \oplus Y_n\, + H'_n,\; { \rm with}\, H'_n \in \cH\,.
 $$ By letting $n \rightarrow + \infty$ and applying Lemma \ref{lem:cont} (for $\widehat{\cF}$), one obtains \eqref{eq:closcong}.
Now, suppose by the contrary that the congruence $\overline{\sim_\cF}$ is trivial. This means that 
$\overline{\cH}= \widehat{\cF}$. In particular, $  E \in \overline{\cH}$ and, by definition of the metric $r(X-Y)$, there exists $H \in {\cH}$ such that:
$$
-\frac{1}{10} E \leq H-E \leq \frac{1}{10} E\,.
$$ This implies that $0 \leq  \frac{9}{10} E \leq H$. Then, by Lemma \ref{lem:cong}.3], $\frac{9}{10} E \in \cH$ and 
by Proposition \ref{prop:alg}.2],  $ t E \in \cH$ for any $t \in \Q$. By the remark following Assumption \ref{ass:c}  and Lemma \ref{lem:cong}.3] again, one then deduces 
that each $X $ of $\cF$ belongs to $\cH$. So $\sim_\cF$ is trivial, and by Proposition \ref{prop:corr}, 
$\sim$ is also trivial which is a contradiction. Therefore, we have proved that $\overline{\sim_\cF}$ is not trivial.

One then obtains the result by applying  Proposition \ref{prop:=} to $\overline{\sim_\cF}$,   and observing that any  $\phi \in  S_E(\widehat{\cF})$ induces 
by restriction an element of $S_E(\cR)$. 
\end{proof}
 
\s

\subsection{ Banach semi-rings. An analogue of Gelfand-Mazur's Theorem.} $\;$

\m
\s
The definition of Banach semi-ring of characteristic $1$ is quite similar to the one for semi-field given 
in Section 2.1. More precisely, 
in the rest of this Section we shall also make the following:
\begin{assumption} \label{Ass:comp}
 $(\cR, \oplus, +)$ is a commutative perfect Banach semi-ring of characteristic $1$, namely  it is complete 
with respect to the distance defined by $r(X-Y)$.

\end{assumption}
By using Cauchy sequences, the action of $\Q^{+*}$ on $(\cR, \oplus, +)$ described in Lemma \ref{lem:alain} is extended naturally to one of $\R^{+*}$ (see \cite{C}). Then this action of $\R^{+*}$ is extended to one on $(\cF, \oplus, +)$.
Moreover, 
$\cF$ becomes endowed with a natural structure of $\R-$vector space. By using the material of Section 2.1, one checks 
right away that $X \mapsto r(X)$ defines a norm of real vector space on $\cF$.
 But in general $\cF$ is not a Banach semi-field, indeed consider 
the example of the Banach semi-ring of all the convex functions $[0;1] \rightarrow \R^+$. We shall denote by 
$(\widehat{\cF}, \oplus, +)$ the Banach semi-field completion of $\cF$.

\s
One defines the concept of maximal cancellative congruence of $\cR$ (in the class of the cancellative congruences) in a similar way as in Definition \ref{def:order}.
Then we state and  prove the following analogue of the Gelfand-Mazur Theorem in the realm of Banach semi-rings.
\begin{theorem} \label{thm:gmring} Let $\sim$ be a maximal cancellative congruence of the Banach semi-ring $\cR$, and denote by 
$\pi_0: \cR \rightarrow \cR / \sim$ the projection.
 \item 1] Suppose that every element $X_1$ of $\cR /\sim$ satisfies $0 \leq X_1$.
 Then $\cR /\sim = \{ \lambda \pi_0(E) /\, \lambda \in \R^+\}$ and one has  an isomorphism of semi-rings:
  $$( \cR / \sim\, , \oplus, +)  \rightarrow (\R^+, \max ( , ), +)\,,$$
 $$ \lambda \pi_0(E) \mapsto \lambda\,.
 $$
 \item 2] Suppose that there exists an element $X_1$ of $\cR /\sim$ which does not satisfy $0 \oplus X_1 = X_1$. Then one has an isomorphism of semi-rings (actually semi-fields):
 $$
 ( \cR / \sim\, , \oplus, +)  \rightarrow (\R, \max ( , ), +) \,.
 $$
\end{theorem}
\begin{proof} By Proposition \ref{prop:corr}, $\sim_\cF$ is a maximal congruence of $\cF$. Denote by $\cH$ the class of $0$ 
for $\sim_\cF$ in $\cF$. 
A priori, $\cH$ is only a $\Q-$vector space.
\begin{lemma} $\cH $ is a closed subset of $\cF$ (endowed with the norm $r(X)$). In particular,  $\cH $ is a $\R-$sub vector space of $\cF$ and $\cF/ \sim_\cF$ 
is naturally endowed with the structure of a $\R-$vector space.
\end{lemma} 
\begin{proof} By applying Lemma \ref{lem:cong}.2] to $\cH$ and using sequences, one checks that 
the closure $\overline{\cH}$ of $\cH$ in $\cF$ satisfies the conditions of Lemma \ref{lem:cong}.2]. Therefore, one defines 
a congruence $\overline{\sim_\cF}$ on $\cF$  by saying that $X \, \overline{\sim_\cF} \, Y$ iff $X-Y \in \overline{\cH}$. 
By maximality of $\sim_\cF$, either $\overline{\sim_\cF} = \sim_\cF$ or $\overline{\sim_\cF}$ is trivial. Suppose, by the contrary, 
that $\overline{\sim_\cF}$ is trivial, in other words, $\overline{\cH}= \cF$. Then proceeding as in the proof of Corollary \ref{cor:ked} 
one obtains that for all $t \in \Q$, $ t E \in \cH$. By Lemma \ref{lem:cong}.3] and Assumption \ref{ass:c}, this  implies that $\cH= \cF$. 
But this contradicts the maximality of $\sim$ and $\sim_\cF$. So, $\cH$ is closed and the rest of the Lemma follows easily.
\end{proof}
Now, denote by $\widehat{\cH}$ the closure of $\cH$ in $\widehat{\cF}$. Proceeding as in the proof of the previous Lemma,  
 one first defines a congruence $\sim_1$ on $\widehat{\cF}$ by saying that $X \sim_1 Y$ iff
$X- Y \in \widehat{\cH}$. Second, one checks (just as before) that  $\widehat{\cH}$ cannot be equal to $\widehat{\cF}$. 
Therefore, by Lemma \ref{prop:rel},  $\sim_1$ is a non trivial closed congruence on $\widehat{\cF}$. By Proposition \ref{prop:=}  applied 
to $\widehat{\cH}$ and $\sim_1$, there exists a character $\phi \in S_E( \widehat{\cF})$ such that 
$X\, \sim_1 \, Y $ implies $\phi(X) = \phi(Y)$. Denote by  $\pi: \cF \rightarrow \cF/ \sim_\cF$  the ($\R-$linear) projection and recall that by Lemma \ref{lem:char}.2], $\phi$ is $\R-$linear.
Since $\cH \subset \widehat{\cH}$, the restriction of $\phi$ to $\cF$ induces 
a $\R-$linear semi-field homomorphism: 
$$\chi: (  \cF / \sim_\cF , \oplus, +) \rightarrow (\R, \max ( , ), +) \, ,$$ such that $\chi \circ \pi = \phi_{| \cF}$ and 
$\chi (\pi (E))=1$. In particular, $\chi$ is not trivial.
By    the maximality of  $\sim_\cF$, $\chi$ is necessarily an injective homomorphism,  it  then  induces 
an isomorphism $\cF / \sim_\cF \, \simeq \,\R$. 

Now consider $X,Y \in \cR $ such that $X \sim Y$. 
Using  Proposition \ref{prop:corr}.1] and the fact that $\cH $ is a $\R-$sub vector space of $\cF$, one obtains that
for any $ \lambda \in \R^+$, $ \lambda X \sim_\cF  \lambda Y$ and thus  $ \lambda X \sim \lambda Y$. Therefore 
$\R^{+*}$ acts naturally on $\cR / \sim \,$   so that the projection $\pi_0: \cR \rightarrow \cR / \sim $ becomes
 $\R^{+*}-$equivariant  and $ \cR / \sim $ contains the half line $\R^+ \pi_0(E)$. 

But, by Proposition \ref{prop:corr}.3], 
$\cF / \sim_\cF\; ( \simeq \R )$ is the semi-field of fractions of $ \cR / \sim $.  
Therefore, if $$ \cR / \sim \, \, \subset \{ X_1 \in \cF / \sim_\cF\,, \;  0 \leq X_1\}\,( \simeq \R^+)\,,$$ then we are in the situation of Part 1]. 
If not, we are in the case of Part 2]. 
One then obtains right away the Theorem.
\end{proof}


%
The example of the Banach semi-ring of all the convex functions $f: [0, 1] \rightarrow \R$ such that $f(1) \geq 0$ 
show that in the previous theorem the two cases indeed can occur. The congruence defined by  $f(1)= g(1)$ 
corresponds to the first case whereas the congruence defined by $f(1/2)=g(1/2)$  corresponds to the second one.

\s 
We end this Section with a special case.
The following Theorem gives a precise characterization of the Banach semi-rings $\cR$ which 
coincide with the non negative part of their semi-field of fractions.

\begin{theorem} \label{thm:banring} Assume moreover that $\cR= \{ X \in \cF / \, 0 \leq X\}$. Then:
\item 1]   The semi-field  of fractions $\cF$ of $\cR$, is a Banach semi-field with respect to the distance  $(X,Y) \mapsto r(X-Y)$.
\item 2] With the notations of Theorem \ref{thm:Tilby}, one has an isometric isomorphism of semi-rings:
$$
\Theta: (\cR, \oplus, +)  \rightarrow (C^0(S_E(\cF), [0, +\infty[), {\max}\,( ,), +)
$$
$$
X \mapsto \Theta_X: \phi \mapsto \phi(X)= \Theta_X (\phi)
$$  
$$\forall X \in \cR, \; r(X)= \sup_{ \phi \in S_E(\cR)} | \Theta_X (\phi) |\,.
$$ 
\end{theorem}
\begin{proof} 1]  
Recall that $\cF$ is perfect since $\cR$ is.
 Consider now  a Cauchy sequence $(X_n)_{n \in \N}$ of points of $\cF$. There exists a sequence of positive rational numbers
 $(t_n)$ such that $\displaystyle \lim_{n \rightarrow + \infty} t_n = 0$ and:
 \begin{equation} \label{eq:eps}
 \forall n, p  \in \N , \; r(X_n - X_{n+p})\, \leq \, t_n / 2\, \,.
\end{equation}
 
 Set for all $n \in \N:$
 $$
 A_n = 0 \oplus X_n ,\; B_n=   0 \oplus ( -X_n) \,.
 $$  The equality $\cR= \{ X \in \cF / \, 0 \leq X \}$ shows that   $(A_n)_{n\in \N}$ is a sequence of points of $\cR$. 
 The inequality  \eqref{eq:eps} implies that $ X_n \oplus (X_{n+p} + t_n E)= X_{n+p} + t_n E$. By applying 
 $\oplus 0 \oplus t_n E$ to this equality one obtains:
 $$
0 \oplus X_n \oplus (X_{n+p} + t_n E) \oplus t_n E = 0 \oplus ( X_{n+p} + t_n E ) \oplus t_n E\,. 
 $$ Using the distributivity of $+$, one sees that this last equality implies: 
 $$
 0 \oplus X_n \oplus  \bigl( ( 0 \oplus X_{n+p}) + t_n E \bigr) = 0 \oplus \bigl( ( 0 \oplus X_{n+p}) + t_n E \bigr) = 
  ( 0 \oplus X_{n+p}) + t_n E \,.
 $$ Indeed, both $0 \oplus X_{n+p}$ and $t_n E$ are $\geq 0$.
 In other words, one has $ A_n \leq A_{n+p} + t_n E$. Similarly, one obtains 
 $-t_n E +A_{n+p} \leq A_n$. To summarize, we have proved that $r( A_n - A_{n+p} ) \leq t_n$. 
 Therefore, the Cauchy sequence $(A_n)$ converges to $ A \in \cR$. Similarly, 
 the sequence $(B_n)$ converges to $ B \in \cR$. This implies that $(X_n)$ converges to 
 $A-B \in \cF$. This proves that $\cF$ is a Banach semi-field. 
 
 
 2] Since $\cR= \{ X \in \cF / \, 0 \leq X \}$, the result is
 a consequence of Theorem \ref{thm:Tilby} applied to $\cF$ and of Proposition \ref{prop:cons}.1].  
 
 \end{proof} 
\begin{remark}  Of course the weaker hypothesis $\cR  \subset \{ X \in \cF / \, 0 \leq X\,\}$  is not sufficient for the  previous Theorem to hold.  Let us mention  two   counter-examples.  Consider  the semi-ring of   all the convex functions  
from $[0,1]$ to $\R^+$  or  the one of all the continuous functions $f: [0,1] \rightarrow \R^+$ satisfying $\forall t \in [0,1], \,  f(1/2) \leq f(t)$.
\end{remark}

\s 
The next Proposition shows that for such  Banach semi-rings the  correspondence 
$\sim \, \mapsto \sim_\cF$ of Proposition \ref{prop:corr}   extends nicely to the closure. Since we shall not use it in  this paper, we leave the
 proof to the reader.

\begin{proposition} \label{prop:cl} Let $\sim$ be a cancellative congruence of $\cR = \{ X \in \cF / \, 0 \leq X \}$. 
 Then the closure $\overline{\sim}$, in $\cR \times \cR$, of $\sim$ 
defines also a cancellative congruence. With the notations of Proposition \ref{prop:corr}, one 
has $\overline{\sim}_\cF\, = \, \overline{\sim_\cF}$.

 Therefore, the congruence $\sim$ is closed if and only if $\sim_\cF$ is  a closed congruence of $\cF$.
\end{proposition} 
\m 



\section{ Foundations for a new theory of schemes in characteristic $1$. }$\;$


\s

  We fix $(\cR, \oplus, +)$ a  commutative perfect cancellative semi-ring of characteristic $1$.  For instance, $\cR$ 
could be the semi-ring $\cR_c( [0,1])$ of all the  piecewise affine convex functions  $[0,1] \rightarrow \R$, 
 its semi-field of fractions being denoted $\cF_c ([0,1])$. We could also consider the semi-ring $\cC_K $ of Example 
\ref{ex:nf}.

In this Section  we adopt a more algebraic view point. Our goal is to propose the foundations of a new scheme theory in characteristic $1$ which should allow, in some sense,  to decide whether 
an element of   $\cF$ belongs or not to $\cR$. We would like also to detect sub semi-rings of $\cR$ which have some arithmetic flavor.
A basic example is provided by 
  $\cF_c( [0,1])$, $\cR_c( [0,1])$ and the semi-ring of functions $ x \mapsto \max_{1 \leq j \leq n} (a_j x + b_j)$ 
  where the $a_j, b_j$ belong to a sub field $\K$ of $\R$.
  

 {\bf We shall assume in this Section that   $\cR$ is a semi-ring as in Definition \ref{def:gen}, which  satisfies the following:   }
\begin{assumption} \label{ass:last} 
\item 1] The semi-field of fractions $\cF$  of $\cR$ satisfies Assumptions 1 and 2.
\item 2]  The action of  $\, \Q^{+*}$ on $(\cR, \oplus, +) $ and $(\cF, \oplus, +) $ defined in Lemma \ref{lem:alain} 
extends to $\R^{+*}$ so that $\cF$ becomes a $\R-$vector space. 

\item 3] For any real $t$, $t E \in \cR$.
\end{assumption} 

The requirement in 3] that $-E \in \cR$ is harmless with respect to the goal of this Section. Indeed,  notice  for instance that the constant function {\bf -1} 
belongs to $\cR_c([0,1])$.

\m
\subsection{A Topology of Zariski Type on $S_E (\cF)$.} $\;$

\s

 Recall that a congruence on $\cF$  is the analogue of the notion of Ideal 
 in Ring Theory. 
We shall denote a congruence by $r$ rather than $\sim$ in order to simplify the notations, indeed several of them will involve indices.

\begin{definition} \label{def:V}
Let $r$ be a congruence on $\cF$ and $\pi_r: \cF \rightarrow \cF/r$  the natural projection.
 Denote by $V(r)$ the set of elements $\phi \in S_E (\cF)$ which 
 can be written as  $\phi= \xi \circ \pi_r$ for a suitable homomorphism of semi-fields 
 $\xi: \cF/r \rightarrow \R$. In other words, $\phi \in V(r)$ if and only if
 $$ \forall X, Y \in \cF, \;  X\, r \, Y \, \Rightarrow 
 \phi(X)=\phi(Y) \,.
 $$ Lastly, we shall denote by $[0]_r$ the class of $0$ of $r$ in $\cF$.
\end{definition}

\begin{remark}
 Since by  Assumption \ref{ass:last}.2], any real $t>0$ induces an automorphism $X \mapsto t X$ of $(\cF, \oplus, +)$, one obtains easily  that:
 $$
 \forall (\phi, \lambda, X) \in S_E (\cF) \times \R \times \cF, \;\, \phi ( \lambda X) = \lambda \phi(X)\,.
 $$
\end{remark}

\s

Observe that $V(r)$ is the analogue for $\cF$ 
 of the set $V(\mathcal{{I}})$ of prime ideals of a ring $A$ containing the ideal $\mathcal{I}$ (see \cite{Ha}[page70]).
If $r$ is the trivial (resp. identity) congruence then $V(r) = \emptyset$ (resp. $V(r) = S_E (\cF)$).
\s 
\begin{proposition} Let $r$ be a non trivial congruence on $\cF$, then 
$V(r)$ is not empty.
\end{proposition}
\begin{proof} This is a direct consequence of  Corollary \ref{cor:ked}  applied with $\cF$ instead of $\cR$.
\end{proof}
\s
 
  We now recall two standards operations on the congruences.

\begin{definition} Let $(r_j)_{j\in J}$ be a family of  congruences on $\cF$.
\item 1]  Denote by $ \vee_{_{j\in J}} r_j$ the intersection of all the congruences $r$ satisfying the following condition:
$$
\forall X, Y \in \cF,\; \; \bigl( \exists j \in J,\; X \,r_j \,Y \bigr) \, \Rightarrow X \,r \,Y \,. 
$$
\item 2]  Denote by $\wedge_{_{j\in J}} r_j$ the congruence defined by:
$$
\forall X, Y \in \cF,\; X \,( \wedge_{_{j\in J}} r_j ) \,Y\; \,{\rm iff}\,\, \forall j \in J,\; X \,r_j\, Y\,.
$$
\end{definition}

\s
The following Proposition gives a precise description of  $ \vee_{_{j\in J}} r_j$ and $\wedge_{_{j\in J}} r_j$.

\begin{proposition} \label{prop:vee} Let $(r_j)_{j\in J}$ be a family of  congruences on $\cF$. Then:
\item 1]  The class of $0$ for $ \vee_{_{j\in J}} r_j$ 
is $\sum_{j \in J} [0]_{r_j}$, the sum of the  classes of zero for each  $r_j$. In other words:
$$
\forall X, Y \in \cF,\; X \, \vee_{_{j\in J}} r_j \, Y \, \Leftrightarrow X- Y \in \sum_{j \in J} [0]_{r_j}\,.
$$
\item 2] The class of $0$ for $\wedge_{_{j\in J}} r_j$ is $\cap_{j \in J} [0]_{r_j}$.
\end{proposition}
\begin{proof}
1] Consider a finite sub family of $F$ of $J$ and set $\cH_j= [0]_{r_j}$.
Using Lemma \ref{lem:cong}.2] for each $\cH_j$,  an easy induction and $0 = 0 \oplus 0$ at the end, one checks that 
$\sum_{j \in F} \cH_j$ is stable under the law $\oplus$. In fact, $\sum_{j \in F} \cH_j$ is a sub semi-field of $\cF$.
Now, using again Lemma \ref{lem:cong}.2] for each $\cH_j$ and an induction, one checks  that 
$\sum_{j \in F} \cH_j$ satisfies the conditions of Lemma \ref{lem:cong}.2]. Therefore, $X-Y \in \sum_{j \in F} \cH_j$ 
defines a congruence on $\cF$. By an inductive limit argument, one sees that $X-Y \in \sum_{j \in J} \cH_j$ 
defines a congruence on $\cF$. It is then obvious that this congruence coincides with $\vee_{_{j\in J}} r_j$.

2] is left to the reader.

\end{proof}

\m
The next  proposition establishes, in characteristic $1$,   the analogue of the following two identities (\cite{Ha}[page 70]) 
for ideals of Ring theory:
$$ \cap_{j \in J} V(\mathcal{I}_j)\,=\, V ( \sum_{j\in J} \mathcal{I}_j)  ,\; \; V(\mathcal{I}_1) \cup V(\mathcal{I}_2) \,=\, V( \mathcal{I}_1 \, \mathcal{I}_2)\,. 
$$ Notice that the proof of the second identity becomes  harder in this setting. Indeed, there is no product in our context.

\begin{proposition} \label{prop:Vcap} \item 1] Let $(r_j)_{ j\in J}$ be a  family  of  congruences on $\cF$. Then:
$$
\cap_{j \in J} V(r_j)\,=\, V (\vee_{_{j\in J}} r_j )\,.
$$
\item 2] Let $(r_l)_{l\in L}$ be a  finite family of congruences on $\cF$. Then:
$$
\cup_{l \in L} V(r_l)\,=\, V( \wedge_{_{l\in L}} r_l)\,.
$$
\end{proposition}
\begin{proof} 1] is an immediate consequence of part 1] of the previous Proposition.

\noindent 2] We prove the result for $L=\{1,2\}$, the general case following by an easy induction on Card\,$ L$.
First we prove that $V(r_1) \cup V(r_2) \subset V (r_1 \wedge r_2)\,.$ 

Denote by $P_1: \cF/( r_1 \wedge r_2)  \rightarrow \cF/r_1$ the projection. Clearly we have: 
$P_1 \circ \pi_{r_1 \wedge r_2 } = \pi_{r_1}$. 
Consider now  $\phi \in V(r_1)$, by definition one has $\phi = \xi_1 \circ \pi_{r_1}$, so that:
$$
\phi= (\xi_1 \circ P_1) \circ \pi_{r_1 \wedge r_2 }\,.
$$ Therefore, setting $\xi= \xi_1 \circ P_1$, one gets that $V(r_1) \cup V(r_2) \subset V (r_1 \wedge r_2)\,.$

Now we prove the converse. So, consider $\phi \in V (r_1 \wedge r_2)\,. $ 
Suppose, by contradiction, that 
$\phi \notin V(r_1) \cup V(r_2)$. Therefore there exists $X,Y,X',Y' \in \cF$ such that 
$X \, r_1 \, Y$, $X' \, r_2 \, Y'$ but:
\begin{equation} \label{eq:no}
\phi(X) \not= \phi(Y), \,\;  {\rm and}\,  \; \phi(X') \not= \phi(Y')\,.
\end{equation} Recall that $\phi$ takes real values.
Now, using the  Definition \ref{def:congF}  of congruences,  we observe that: 
$$
\Bigl( (X+X') \oplus (Y+Y') \Bigr) \,  r_1 \wedge r_2  \, \Bigl( (Y+X') \oplus (X+Y') \Bigr)\,.
$$ But \eqref{eq:no} immediately implies that
$$
 \phi(X) + \phi(X')\, ,  \, \phi(Y) + \phi(Y')  \,{\bf \notin }\, \{ \phi(Y) + \phi(X')\,,\, \phi(X) + \phi(Y') \}\,.
$$
This shows that:
$$
\phi \Bigl( (X+X') \oplus (Y+Y') \Bigr) = \max ( \phi(X) + \phi(X')\,,\, \phi(Y) + \phi(Y') )
$$ cannot be  not equal to
$$
\phi \Bigl( (Y+X') \oplus (X+Y') \Bigr) = \max ( \phi(Y) + \phi(X')\,,\, \phi(X) + \phi(Y') )\,.
$$ 
This contradiction of the definition of $\phi \in V (r_1 \wedge r_2)$ shows that 
$$
V (r_1 \wedge r_2) \subset V(r_1) \cup V(r_2)\,.
$$ The result is proved.
\end{proof}

\s
\begin{definition} \label{def:zar} By the previous Proposition,     the $V(r)$, where 
$r$ is any congruence on $\cF$, define the closed subsets of a topology  on 
$S_E (\cF)$. We call it the topology $\cZ$ of Zariski. 
 
 \end{definition}
 
 \s
Recall that in Definition \ref{def:topo} we introduced another topology, called $\cT$, on $S_E (\cF)$. The following Proposition compares them.
\begin{proposition}  The identity map: $(S_E (\cF), \cT) \rightarrow (S_E (\cF), \cZ)$ 
is continuous. Therefore, $(S_E (\cF), \cZ)$ is quasi-compact.
\end{proposition}
\begin{proof} Consider $\phi \in V(r)^c$, the complementary of a closed subset associated to a congruence $r$.
By definition, $\exists X, Y \in \cF$ such that $X \, r \, Y $ and $| \phi(X) - \phi (Y) | = 5 \delta >0$.
It is clear that if $\phi' \in S_E(\cF)$ satisfies $| \phi(X)- \phi'(X) | < \delta$ and 
$| \phi(Y)- \phi'(Y) | < \delta$ then $| \phi'(X)- \phi'(Y)| > \delta >0$. This shows that 
$V(r)^c$ contains an open neighborhood for $\cT$ of each of its point $\phi$. Therefore 
$V(r)^c$ is an open subset for $\cT$. Now, since $(S_E (\cF), \cT)$ is compact (Theorem \ref{thm:comp}), we obtain 
the whole proposition.
\end{proof}

If $\cF$ were a Banach semi-field, we could use Theorem \ref{thm:Tilby} and Urysohn to show 
that $(S_E (\cF), \cZ)$ is Hausdorff. But in general we do not know whether $(S_E (\cF), \cZ)$ is, or not, Hausdorff.

\s 
Next we prove the  analogue of the following well know fact for ideals $I_j$ in a ring $A$.
If $\cap_{j \in J} \, V(I_j)\,= \emptyset $ then there exists a finite sub family $F$ of $J$ 
such that $\sum_{i\in F} \, I_i \, = A$.

\begin{corollary} Let $(r_j)_{j\in J}$ be a family of  congruences of $\cF$. Assume that 
$\displaystyle \cap_{j \in J} \, V(r_j)\,= \emptyset $. There there exists a finite sub family $F$ of $J$ 
such that $\displaystyle \sum_{i \in F}\,  [0]_{r_i} = \cF$.

\end{corollary}
\begin{proof} By the previous Proposition, $(S_E (\cF), \cZ)$ is quasi-compact, so there exists a finite sub family $F$ of $J$ such that 
$\cap_{i\in F} \, V(r_i)\,= \emptyset $. Set $\cH = \sum_{i\in F} \, [0]_{r_i}$. 
By parts 1] of Propositions \ref{prop:vee} and \ref{prop:Vcap}, 
 one defines 
a congruence $\sim$ on $\cF$ by saying that $X  \sim Y $ iff $X - Y \in \cH$,  and 
there is  no element $\phi \in S_E(\cF)$ which vanishes on $\cH$. Then, applying Corollary \ref{cor:ked} with 
$\cF$ instead of $\cR$, one deduces that $\cH= \cF$. 
The result is proved.
\end{proof}

\s
Now let us examine the action of a homomorphim of semi-fields  on these topologies.
\begin{proposition} \label{prop:zar} Let $\Phi: (\cF, \oplus, +) \rightarrow (\cF', \oplus, +)  $ be a homomorphism between two semi-fields satisfying Assumption \ref{ass:last}, such that $\Phi(E)=E$.
\item 1] The map $F: \phi \mapsto \phi \circ \Phi$ is continuous from $(S_E (\cF'), \cZ) $ to $(S_E (\cF) , \cZ)$. 

\item 2] The map $F$ is also  continuous from 
 $(S_E (\cF') , \cT) $ to   $(S_E (\cF) , \cT)$. 
\end{proposition}
\begin{proof} 1] Consider a congruence $r$ on $\cF$. Denote by $r'_\Phi$ the intersection of all the congruences $r'$ of 
$\cF'$ satisfying the condition:
$$
\forall X, Y \in \cF, \; X \,r\, Y \, \Rightarrow \, \Phi(X) \,r'\, \Phi (Y)\,.
$$ Let us show that $F^{-1} (V(r))= V(r'_\Phi)$. 

 Let $\phi \in S_E (\cF')$, denote by $r'_\phi$ the congruence on $\cF'$ defined  by: $A \, r'_\phi \, B$ if $\phi(A) = \phi(B)$.
 Then $\phi \in F^{-1} (V(r))$ or, equivalently
 $\phi  \circ \Phi \in V(r)$, if and only if:
 $$
 \forall X, Y \in \cF, \; X \,r\, Y \, \Rightarrow \, \Phi(X) \, r'_\phi \, \Phi(Y)\,.
 $$ But this means exactly that: 
 
 $$ \forall A, B \in \cF', \;  A\, r'_\Phi \, B \, \Rightarrow  A\, r'_\phi \, B \,.
 $$ In other words, this means that $\phi \in V(r'_\Phi)$. Therefore, $F^{-1} (V(r))= V(r'_\Phi)$, which proves the 
 continuity of $F$.
 2] is left to the reader.
\end{proof}

\s

\subsection{ Valuation and localization for semi-rings. } $\;$

\m 

We first motivate the next Definition by analyzing a simple example. Denote by $\cF_c([0,1])$ the set of  piecewise affine functions on $[0,1]$, this is
the semi-field 
of fractions of the semi-ring $\cR_c([0,1])$ introduced at the beginning of this Section. 
Consider $x \in ]0,1[$ and  $X \in \cF_c([0,1])$ such that $X(x)=0$. 
Then the following number 
$$ 
    \frac{X( x-h) + X(x+ h)}{h}  \,,
$$ does not depend on $h>0$ for $h$ small enough: it is equal to $X'_d(x) - X'_g(x)$. We denote it $V_x( X)$. 
Observe that for any $X, Y \in \cF_c([0,1])$ with $X(x)=Y(x)=0$, one has:
$$
V_x( X+Y)= V_x(X)+V_x(Y),\; \max( V_x(X) , V_x(Y) ) \leq  V_x ( \max (X, Y) )\,.
$$
Moreover, for any $A \in \cF_c([0,1])$, $A$  belongs to $\cR_c([0,1])$ if and only if

\begin{equation} \label{eq:convex} \forall x \in ]0,1[, \, 0\leq V_x ( A - A(x) E)\,,
\end{equation}

where $E$ is the constant function {\bf 1}. As for the extreme points of $[0,1],$ we set $V_0=V_1=0$.


\s 

{\it Recall that in this Section, $\cR$ denotes a semi-ring of characteristic $1$ satisfying Assumption \ref{ass:last}.}
In order to be able to generalize properly  these observations 
to $\cR$, we first define the notion of valuation 
in our context.

\begin{definition} \label{def:val} 
A valuation on $\cR$ is given by a pair
$(\phi , V_\phi)$ where $\phi \in S_E (\cF)$ and 

\noindent $V_\phi: \phi^{-1}\{0\} \cap \cR \rightarrow  [0, + \infty[$ 
is a map satisfying the following two conditions.

\item 1]  For all $X,Y  \in \phi^{-1}\{0\} \cap \cR$,  $V_{\phi} (X +Y)= V_{\phi} (X) + V_{\phi}(Y) \,.$

\item 2]  For all $X,Y  \in \phi^{-1}\{0\} \cap \cR$, $ \max ( V_{\phi} (X) , V_{\phi} (Y) )\, \leq \, V_{\phi} (X \oplus Y) \,.$

One defines similarly a valuation on the semi-field of fractions $\cF$ of $\cR$. It is  a pair $(\phi , V_\phi)$ where 
$\phi \in S_E(\cF)$ and $V_\phi: \phi^{-1}\{0\} \rightarrow \R$ is a map 
satisfying the same two previous conditions but with $X,Y \in \phi^{-1}\{0\}$.

\end{definition}

We insist on the fact that by Definition, {\bf a valuation on a semi-ring $\cR$ is defined only on $\phi^{-1}\{0\} \cap \cR$ and  takes positive values or zero.}
This said,  a valuation $V_\phi$ on $\cF$ takes values in $\R$ and, morally 
 $V_{\phi} (X- \phi (X) \, E)$ defines an order of vanishing. 

\s
Let us give several concrete examples of this concept of valuation.

\begin{example} \item 1] Consider $\cR_c([0,1])$, the semi-ring of all the piecewise affine convex functions on $[0,1]$.
Define a character $\phi$ by  $\phi(X)= X(1/2)$ for $ X \in \cR_c([0,1])$ and, if  $X(1/2)=0$ set:
$$
V_\phi (X) = X_d'(1/2)- X_g'(1/2)\,.
$$
Another valuation $V^1_\phi$ is defined by $V^1_\phi (X) = \displaystyle \frac{ X(2/5) + X(3/5)}{2}$.
\item 2] Consider the semi-ring $\cR$ of all the functions $X(x,y)= \max_{1\leq j \leq n} ( a_j x + b_j y + c_j) $ 
on the square $[-1, 1]^2$. Define  a character $\phi$ by $\phi(X)= X(0,0)$. If $X(0,0)=0$, then the integral on 
the disc $D(0,r)$:
$$
\frac{1}{r^3} \int_{D(0,r)} X(x,y) dx\, d y
$$ does not depend on $r>0$ for $r$ small enough. Call it $V_\phi(X)$. This defines a valuation on $\cR$.
\end{example}

\s
The next Lemma explains how a  valuation on $\cR$ can be uniquely extended to a  
valuation on its semi-field of fractions $\cF$. 
\begin{lemma} \label{lem:vallin} Let $(\phi , V_\phi)$ be a valuation on $\cR$. 
\item 1] One defines a valuation 
$( \phi , V'_{\phi})$ on $\cF$ in the following way. For any $A, B \in \cR$ such that $\phi(A-B)=0$, set:
$$
V'_{\phi} (A-B)= V_\phi (A- \phi(A) E) - V_\phi (B- \phi(B) E)\,.
$$ Moreover, $( \phi , V'_{\phi})$ is the unique valuation of $\cF$ whose restriction
to $\cR$ gives $(\phi , V_\phi)$. Therefore, we shall frequently write $(\phi , V_\phi)$ instead of $( \phi , V'_{\phi})$.

\item 2] Let  $X \in \cF$ be  such that $\phi(X)=0$. Then for any $t\in \R$,  one has $V'_\phi ( t X) = t V'_\phi ( X)$.
\end{lemma}
 
 \begin{proof}
 1]   Notice that by Assumption \ref{ass:last}, $A- \phi(A) E \in \cR$ for all $A \in \cR$.
 Consider then $A,B,A',B' \in \cR$ such  that  $A-B=A'-B'$ and $\phi(A-B)=0$.
 Using the additivity of $V_\phi$, one checks easily that:
 $$
 V_\phi (A- \phi(A) E) - V_\phi (B- \phi(B) E) \, = \, V_\phi (A'- \phi(A') E) - V_\phi (B'- \phi(B') E) \,.
 $$ Therefore, $V'_{\phi}$ is well defined. 
 
 It is easily checked that $V'_{\phi}$ satisfies, on $\phi^{-1}\{0\}$, 
 the additivity condition 1] of Definition \ref{def:val}. Let us prove that it satisfies also the condition 2].
 So, consider $A,B,A_1,B_1 \in \cR$ such that $\phi(A-B)=0 =\phi(A_1-B_1)$. We can write:
 $$
 V'_\phi ( (A-B) \oplus (A_1 - B_1) )= M\,-\, 
 V_\phi \bigl(B+ B_1 - \phi(B+B_1) E\bigr) \,, 
 $$ where we have set:
 $$
 M= V_\phi \bigl( \,(A+B_1) \oplus (A_1 + B) - \phi(B+B_1) E  \, \bigr)\,.
 $$ 
 Next, using  Definition \ref{def:val}.2] for $V_\phi$, we obtain:
 $$
 M= V_\phi \bigl( \,(A+B_1 - \phi(B+B_1) E ) \oplus (A_1 + B - \phi(B+B_1) E )  \, \bigr)\geq 
 $$
 $$\max \bigl( \, V_\phi (A+B_1 - \phi(B+B_1) E ) ,   
 V_\phi (A_1 + B - \phi(B+B_1) E) \, \bigr)\,.$$
 Now, by subtracting $V_\phi \bigl(B+ B_1 - \phi(B+B_1) E\bigr) $ to 
 both sides of this inequality, we obtain:
 $$
 V'_\phi ( (A-B) \oplus (A_1 - B_1) ) \geq \max \bigl( \, V'_\phi  (A-B) , V'_\phi  (A_1-B_1) \, \bigr)\,.
 $$ The result is proved.
 
 2]  Write $X = X_+ - X_-$ where $X_+ = 0 \oplus X$, $X_- = 0 \oplus (-X)$. Observe that $\phi ( X_\pm)=0$.
Consider  the  map $t \mapsto \psi (t)= V'_\phi( t X_+)$ defined from $\R$ to $\R$.
By the definition of $V'_\phi$, for any reals $t,t'$ one has: $\psi( t+ t')= \psi(t) + \psi(t')$. 
Consider now $(t,h) \in \R \times \R^+$. Since $0 \leq X_+$, Assumption \ref{ass:last}.2] allows to check that $t X_+ \oplus (t+h) X_+ = (t+h) X_+$. 
By the definition of $V'_\phi$, one then has:
$$
\max( V'_\phi( t X_+) , V'_\phi( (t+h) X_+) ) \, \leq \, V'_\phi( t X_+ \oplus (t+h) X_+)  = V'_\phi((t+h) X_+)\,.
$$ In other words, $\psi$ is nondecreasing on $\R$. It is then standard that for all $t \in \R$,
$\psi(t)=t \psi(1)$, so that $V'_\phi( t X_+)= t V'_\phi(  X_+)$. Now, since $V'_\phi(-Z)=- V'_\phi(Z)$
for any $Z \in \phi^{-1}\{0\}$, one obtains easily the result.

 \end{proof}

Notice the following subtle point in the previous Lemma. Even if  $V'_{\phi} (X - \phi(X) E) =0$ for an element $X $ of $\cF$, then it is not true in general that
$X$ belongs to $\cR$.

\s
Now we come to the concept of $v-$local semi-ring.

\begin{definition} \item 1] A  semi-ring $\cR$ is called  v-local if it is endowed with a valuation $(\phi , V_\phi)$
such that any $X \in \cR$ satisfying $V_\phi ( X - \phi (X) E) = 0$ is invertible in $(\cR, +)$.

\item 2] A  semi-ring $\cR$ is said to be of local valuation if $\cF$ is endowed with a valuation $(\phi , V_\phi)$
such that $\cR = \{ X \in \cF/\; 0 \leq  V_\phi ( X - \phi (X) E) \}$.
\end{definition}

A local valuation semi-ring is clearly $v-$local, but the converse is false.

\s
Next we define the localized of a semi-ring $\cR$ along $(\phi, V_\phi)$ by analogy with the process of localization 
in classical commutative algebra.

\begin{definition} \label{def:loc} Let $\cR$ be a semi-ring (as in Assumption \ref{ass:last}) endowed with a valuation 
$(\phi, V_\phi)$.  We define the localization $\cR_{[\phi]}$ 
of $\cR$ at $\phi \in S_E (\cF)$  to be the following sub semi-ring of $\cF$:
\begin{equation} \label{eq:loca}
 \cR_{[\phi]}= \{ A-B/ \, A, B \in \cR ,\,    \; V_{\phi} (B - \phi(B) E)=0 \}\,.
\end{equation}
The valuation $(\phi, V_\phi)$ extends naturally to a valuation (denoted by  the same name) of 
$\cR_{[\phi]}$ and makes it a $v-$local semi-ring.
\end{definition} 

That $\cR_{[\phi]}$ is indeed a $v-$local  semi-ring  is an easy exercise. 

 One motivates the previous definition
as follows. In standard algebraic geometry, the idea of localization is to make invertible the regular functions 
which do not have a zero at $\phi$. For the tropical  semi-field $\cF_c([0,1])$ above, 
a zero (resp. pole)  of the piecewise affine function $B$ is a point $x=\phi $ such that $0 < B'_d(x) -B'_g(x)$
(resp. $B'_d(x) -B'_g(x) <0$). 
With the notations of \eqref{eq:convex},  the condition $V_{\phi} (B- \phi(B)E) =0$  means precisely that $x$ is neither a zero nor a pole in the tropical sense.

\s

\s
Now we come to the analogue of the concept of morphisms between local rings in algebraic geometry.
\begin{definition} \label{def:ml} A local morphism  between two $v-$local semi-rings $(\cR', \phi', V_{\phi'})$ and 
$(\cR, \phi, V_{\phi})$ is given by a morphism of semi-rings $\Phi : \cR' \rightarrow \cR$ such that:

\item 1] $\Phi(E)= E$ and $\phi\circ \Phi = \phi'$.

\item 2] For any $ X \in \cR'$,
$ V_{\phi'}( X - \phi'(X) E) > 0 \, \Leftrightarrow \,  V_{\phi} ( \Phi(X) - \phi \circ \Phi(X)\, E \,) >0$

\end{definition}

\s
%

%
%

\m

\subsection{Locally semi-ringed spaces. Schemes.} $\;$

\s
Next, we define the analogue, in characteristic $1$, of the notion of locally ringed space 
in Algebraic Geometry (\cite{Ha}). 
\begin{definition} A $v-$locally semi-ringed space $(S, \cO)$  is given by a topological space $S$ 
endowed with a sheaf of semi-rings $\cO$ satisfying the following conditions.

\item 0] For any inclusion $U \subset W$ of open subsets of $S$, $\cO( W)$ satisfies Assumption \ref{ass:last} and
the restriction map $\cO( W) \rightarrow \cO(U)$ sends $E$ to $E$.

\item 1] For each $s \in S$, the stalk $\cO_s$ of $\cO$ is endowed 
with a valuation $(\phi_s, V_s)$ making it a $v-$local semi-ring.

\item 2] Let $U$ be an open subset of $S$ and  $s_0 \in U$.  Consider $X, Y \in \cO(U)$ 
 such that their germs $X_{s_0}, Y_{s_0} \in \cO_{s_0}$ satisfy $V_{s_0} ( X_{s_0} - \phi_{s_0} (X_{s_0}) E) = V_{s_0} ( Y_{s_0} - \phi_{s_0} (Y_{s_0}) E)$.
Then, there exists a neighborhood $\Omega$ of $s_0$ in $U$ 
such that
$$
\forall s \in \Omega , \; V_{s} ( X_{s} - \phi_s (X_{s}) E) = V_{s} ( Y_{s} - \phi_s (Y_{s}) E)\,.
$$
\end{definition}

\s

The  definition of the  concept of morphism between two $v-$locally semi-ringed spaces $(S, \cO)$ and $ (S', \cO')$   needs 
some preparation. Consider a pair $(F, F^\#)$ where $F: S \rightarrow S'$ is a continuous map and 
$F^\#: \cO' \rightarrow F_* \cO$ is a morphism of sheaves of semi-rings sending $E$ to $E$.

Consider $s \in S$ and an open neighborhood $U$
 of $F(s)$.
 By definition we have a morphism of semi-rings: 
$F^\#(U): \cO'( U) \rightarrow \cO( F^{-1}(U) $. By taking the inductive limit when $U$ runs over the open neighborhoods of $F(s)$,
 we obtain, by the definition of a morphism of sheaves, a homomorphism of semi-rings $F^\#_{s} : \cO'_{F(s)} \rightarrow \cO_s$.

\begin{definition} A morphism between two $v-$locally semi-ringed spaces $(S, \cO)$ and $ (S', \cO')$ is given by a pair 
$(F, F^\#)$ as above satisfying the following condition. For any $s \in S$, the map $F^\#_{s}$ defines a local morphism, in the sense of 
Definition \ref{def:ml},  from 
$( \cO'_{F(s)}, \phi'_{F(s)}, V'_{F(s}) $ to $(\cO_s, \phi_s, V_s)$, where $\phi'_{F(s)} = \phi_s \circ F^\#_{s}$.

\end{definition}

\s Now, in order to be able to associate a $v-$locally semi-ringed space to a semi-ring $\cR$, we need to assume that it has some extra structure.
The set $S_E (\cF)$  is endowed either with the topology $\cT$ or the 
Zariski one (see end of Section 7.1). 

\begin{definition} \label{def:loca} Let $\cR$ be a semi-ring (as in Assumption \ref{ass:last}). We shall call localization data on $\cR$, the data for each $\phi \in S_E (\cF)$,
 of   a  valuation $(\phi,V_{\phi})$ on $\cR$   
 such that  the following condition is satisfied.
 Let $(X , \phi_0) \in \cF \times S_E (\cF)$ be such that $V_{\phi_0 } (X- \phi_0 (X) \, E)=0$.
Then there exists a neighborhood $\Omega$  of $\phi_0$  in $ S_E (\cF)$ such that
$$
\forall \phi \in \Omega, \; V_{\phi} (X- \phi (X) \, E)=0\,,
$$ where we have still denoted $V_\phi$ the canonical extension of $V_\phi$ to $\cF$.
We shall denote such localization data by $(S_E (\cF), (V_\phi) )$.
\end{definition} 








%

The next proposition shows that
these   localization data allow to define a natural semi-ring of $\cF$ in the same way that the data \eqref{eq:convex}  
 characterize the convex functions in $\cF_c([0,1])$. 
\begin{proposition} We keep the notations of the previous Definition. One obtains  a sub  semi-ring of $\cF$ stable by the action of $\R^+$  by setting:
$$
\widehat{\cR}\,=\,\{ X \in \cF / \;  \forall \phi \in S_E (\cF) , \; 0 \leq  V_{\phi} (X- \phi (X) \, E) \}\,.
$$ 
\end{proposition}
\begin{proof}  Definition \ref{def:val}.1] implies that $V_\phi(0)= 0$, it is then clear that 
$\R E \subset \widehat{\cR}$. 

Consider now $X, X' \in \widehat{\cR}$ and let us check that 
$X \oplus X' \in \widehat{\cR}$. Let $\phi \in S_E(\cF)$, assume for instance that $\phi(X') \leq \phi(X)$ so that 
$\phi (X \oplus X')= \phi(X)$ by the properties of a character. Set $Z= X \oplus X' - \phi(X) E$,
then one has:
$$
Z = (X - \phi(X) E) \oplus (X' - \phi(X) E) = (X - \phi(X) E) \oplus Z\,.
$$ Applying then Definition \ref{def:val}.2], we deduce:
$$
V_\phi(Z)=V_\phi \bigl(\, (X - \phi(X) E) \oplus Z \, \bigr) \geq \max \bigl(  V_\phi (X - \phi(X) E) , V_\phi(Z) \bigr) \,.
$$ Since $V_\phi (X - \phi(X) E) \geq 0$, 
 we obtain $V_\phi(Z) \geq 0$, which proves that $X \oplus X' \in \widehat{\cR}$. Moreover, it is clear that 
$X+ X' \in \widehat{\cR}$. Lastly, consider $t\in \R^+$ and $X \in \widehat{\cR}$, then by  Lemma \ref{lem:vallin}.2]  it is clear that $t X \in \widehat{\cR}$.
The Proposition is proved.
\end{proof} 
\s 
Now we define the structural sheaf $\cO_\cR$ on $S_E (\cF)$.   

\begin{definition} \label{def:O} Assume that $\cR$ is endowed with localization data $( S_E (\cF), (V_\phi) )$.
Let $U$ be an open subset of $S_E (\cF)$. An element of $F \in \cO_\cR(U)$ is a function on $U$ which to each  $\phi \in U$ associates 
an element $F(\phi) \in \cR_{[\phi]}$ subject to the following condition. Each point $\phi_0$ of  $ U$ 
admits an open neighborhood $\Omega \subset U$ such that there exists $A,B \in \cR$:
$$
\forall \phi \in \Omega,\; F(\phi) = A-B ,\; {\rm and} \;V_\phi(B - \phi (B) E) = 0\, .
$$  The usual restriction between open subsets endows $\cO_\cR$ with the structure of sheaf of semi-rings. The stalk of $\cO_\cR$ at each point $\phi$ is 
$\cR_{[\phi]}$ (see Definition \ref{def:loc}).
\end{definition}
It is worthwhile to recall that by definition of $(\cR, V_\phi)$, $V_\phi(A - \phi (A) E) \geq 0$.
 
 \s

By combining the Definitions \ref{def:loca} and \ref{def:O} we obtain immediately the following Lemma 
which fixes the concept of affine scheme in our context.

\begin{lemma}  Assume that $\cR$ is endowed with localization data $( S_E (\cF), (V_\phi) )$. 

\noindent Then 
$( S_E (\cF), \cO_\cR)$ is a $v-$locally semi-ringed space of a particular type, which we call affine.
\end{lemma}

\begin{definition} \label{def:scheme} A $v-$locally semi-ringed space is called a scheme if locally it is isomorphic 
to an open subset of an affine scheme.
\end{definition}

\s
Let us give an example of a scheme. 
\begin{example} \label{ex:scheme}
Endow $\R / \Z$ with the usual topology  and with the sheaf $\cO$ of functions $X$ which near each point $s_0$ of $ \R / \Z$ are of the 
following form.  For $s \leq s_0, \, X(s)= a s + b$;  for $s \geq s_0,\; X(s)= a' s + b'$ where:
$$
a,b, a', b' \in \R,\; a s_0 + b = a' s_0 + b',\; 0\leq -a + a' \, = V_{s_0} (X - X(s_0) {\bf 1})\,.
$$ Then $(\R / \Z, \cO)$ is a scheme in the sense of the previous definition.
\end{example}

\s
Now we introduce a concept allowing to detect a structure having some arithmetic flavor.
\begin{definition} Let $(S, \cO)$ be a scheme and $\K$ a sub field of $\R$. Consider a sub sheaf of semi-rings $\widehat{\cO}$ of $\cO$ where, for $U$ open in $S$,  the $\widehat{\cO}(U)$ are not assumed to be preserved by the action of  $\R^{+*}$.
One says that 
 $\widehat{\cO}$  is defined over $\K$ if it satisfies the following conditions.

\item 1] Let $s \in S$ be such that $\phi_s ( \widehat{\cO}_s)$ is not included in $\K$. Then for any 
$X \in \widehat{\cO}_s$, 

\noindent $V_{\phi_s} ( X - \phi_s (X) E) = 0$.

\item 2] If $\phi_s ( \widehat{\cO}_s) \subset \K$ then $\K E \subset \widehat{\cO}_s $, and for any 
$X \in \widehat{\cO}_s$, $V_{\phi_s} ( X - \phi_s (X) E) \in \K \cap [0, + \infty[$.

\item 3] The subset $\{s \in S /\; \phi_s ( \widehat{\cO}_s) \subset \K \}$ is dense in $S$.
\end{definition}

Consider the Example \ref{ex:scheme}. Then the functions $X$ with $a,b,a',b' \in \K$ define a sub-sheaf  $\widehat{\cO}$ of 
$\cO$ which  is defined over $\K$. Indeed, if $s_0$ does not belong to $\K$ then the equality $a s_0 + b = a' s_0 + b'$ 
implies $a=a'$ so that $V_{s_0} (X - X(s_0) {\bf 1})=0$.


\s

\m
We have developed some foundations for  a new scheme theory in characteristic $1$. We shall try to explore further this issue in a future paper.

	\bibliographystyle{amsalpha}

\end{document}